\documentclass[11pt,reqno]{amsart}
\usepackage{amsmath,amsfonts,amssymb,mathrsfs,graphicx,subfigure}
\usepackage{color,psfrag,graphicx}
\usepackage{algorithm,algpseudocode,enumerate}
\usepackage[subnum]{cases}
\usepackage{empheq}

\usepackage{todonotes}
\usepackage[normalem]{ulem}

\usepackage{appendix}

\numberwithin{equation}{section}

\newtheorem{theorem}{Theorem}[section]
\newtheorem{remark}[theorem]{Remark}
\newtheorem{lemma}[theorem]{Lemma}
\newtheorem{example}[theorem]{Example}

\newtheorem{assumption}[theorem]{Assumption}

\usepackage[pdftex,colorlinks,bookmarksopen,bookmarksnumbered,citecolor=red,urlcolor=red]{hyperref}

\newcommand{\tu}{\tilde{u}}
\newcommand{\hu}{\hat{u}}

\def\bU{\boldsymbol{U}}

\newcommand{\RR}{\mathbb{R}}
 \newcommand{\Real}{\mathbb{R}}

\def\bx{\boldsymbol{x}}
\def\bn{\boldsymbol{n}}
\def\bxi{\boldsymbol{\xi}}

\def\btheta{\boldsymbol{\theta}}
\def\btau{\boldsymbol{\tau}}
\def\leq{\leqslant}
\def\geq{\geqslant}

\def\Lo{\mathcal{L}}

\newcommand{\eps}{\varepsilon}

\newcommand{\abs}[1]{\lvert#1\rvert}
\newcommand{\norm}[1]{\lVert#1\rVert}

 \newcommand{\set}[1]{ \left\{#1\right\}}

\newcommand{\rd}{\mathrm{d}}

\newcommand{\dom}{{D}}

\newcommand{\Od}{\mathcal{O}}
\newcommand{\rin}{\mathrm{int}}

\newcommand{\ext}{\mathrm{ext}}

\renewcommand{\vec}[1]{\boldsymbol{#1}}
\DeclareMathAlphabet{\mathsfsl}{OT1}{cmss}{m}{sl}
\newcommand{\tensor}[1]{\mathsfsl{#1}}
\newcommand{\domep}{{D_\eps^+}}
\newcommand{\domen}{{D_\eps^-}}
\newcommand{\Ge}{{\Gamma_\eps}}
\def\bx{\boldsymbol{x}}

\newcommand{\sqbk}[1]{\left[  #1 \right]}

\begin{document}

\title[]
{Two-parameter Asymptotic expansions for elliptic equations with
small geometric perturbation and
  high contrast ratio}

\author[J. Chen]{Jingrun Chen}
\address{Mathematical center for interdisciplinary research and
School of Mathematical Sciences, Soochow University, Suzhou, 215006, China} \email{jingrunchen@suda.edu.cn}

\author[L. Lin]{Ling Lin}
\address{Department of Mathematics, City University of Hong Kong, Tat Chee Ave, Kowloon,
Hong Kong SAR} \email{linling059@gmail.com}

\author[Z. Zhang]{Zhiwen Zhang}
\address{Department of Mathematics, The University of Hong Kong, Pokfulam, Hong Kong SAR} \email{zhangzw@maths.hku.hk}

\author[X. Zhou]{Xiang Zhou}
\address{Department of Mathematics, City University of Hong Kong, Tat Chee Ave, Kowloon,
Hong Kong SAR} \email{xiang.zhou@cityu.edu.hk}

\date{\today}
\begin {abstract}
We consider the asymptotic solutions of an interface problem corresponding to an elliptic partial differential equation with Dirichlet boundary condition and transmission condition, subject to the small geometric perturbation and the high contrast ratio of the conductivity. We consider two types of perturbations: the first corresponds to a thin layer coating a fixed bounded domain and the second is the perturbation of the interface. As the perturbation size tends to zero and the ratio of the conductivities in two subdomains tends to zero, the two-parameter asymptotic expansions on the fixed reference domain are derived to any order after the single parameter expansions are solved beforehand. Our main tool is the asymptotic analysis based on the Taylor expansions for the properly extended solutions on fixed domains. The Neumann boundary condition and Robin boundary condition arise in  two-parameter expansions, depending on the relation of the geometric perturbation size and the contrast ratio.
\end {abstract}

\maketitle


\section{Introduction}

Let $\dom\subset \RR^d$, $d=1,2,3$, be
  a simply-connected Lipschitz continuous domain.
Consider the perturbation of the domain $\dom$ given by
 the perturbed  boundary $\partial \dom_\eps$ defined as  \begin{equation}\label{h}
\partial \dom_\eps=\set{\bx':\bx'=\bx+\eps h(\bx)\bn(\bx):\bx\in\partial \dom},
\end{equation}
where $\eps\in (0,~\eps_0]$ for a fixed small number $\eps_0 \ll 1$ represents the small characteristic size of the perturbation,
$h(\bx)$ is a continuous function defined on $\partial \dom$,
and  $\bn(\bx)$ is  the (outward) normal direction of $\dom$.
For sufficiently small $\eps$, the boundary $\partial\dom_\eps$ uniquely defines a perturbed domain $\dom_\eps$.
If $h$ is non-negative, then    $\dom_\eps$ contains $\dom$.
We assume that $h$ is a sufficiently smooth function.

The main problem  of our concern
is related to the following Dirichlet boundary value elliptic problem imposed in the  perturbed domain $\dom_\eps$:
\begin{equation}\label{eqn:uePDE}
\begin{cases}
&\Lo u_\eps=f \quad \text{in } \dom_\eps,\\
& u_\eps=g \quad \text{on } \partial \dom_\eps,
\end{cases}
\end{equation}
where $\Lo$ is the second order elliptic operator,   having the  divergence  form
\begin{equation}\label{def:Lo}
\Lo u=-\sum_{i,j=1}^d \partial_{x_j}\bigl(a^{ij}(\bx)\partial_{x_i}u\bigr)+\sum_{i=1}^db^i(\bx)\partial_{x_i}u+c(\bx)u.
\end{equation}
The second order coefficient functions $a^{ij}$, $i,j=1,\cdots, d$, form a
 non-degenerate positive definite  matrix $\tensor{a}=(a^{ij})$, i.e., $a^{ji}=a^{ij}$,
 and
\begin{equation}\label{eqn:ellipcond}
\sum_{i,j=1}^d a^{ij}(\bx)\xi_i\xi_j > 0,
\end{equation}
{for every } $\bx\in\overline{\dom\cup\dom_{\eps_0}}$  and  non-zero vector $(\xi_1,\cdots,\xi_d)\in\RR^d$.
The coefficients $b^i$, $i=1,\cdots, d$, and $c$  are assumed smooth in $\Real^d$.
The boundary value function $g$ is also assumed smooth in an open neighbourhood of $\partial\dom$.

If the   coefficient $\tensor{a}$ is assumed to be continuous everywhere,
then the solution $u_\eps$ is the perturbation of a classic elliptic equation with uncertainty
in characterizing the domain.
How to quantify the uncertainty in the solution due to the geometric perturbation, particularly
when $h$ is a random function, is an interesting and important topic in uncertainty quantification.
The more challenging case is that $\tensor{a}$ is not continuous across some interface. Then
the transmission condition should be specified on the jump interface.
In such cases, the interface  may also be subject to small perturbations.

There are two scenarios  of the geometric  perturbations  in the transmission problems.
The first one is to consider the previous domain perturbation setup
with a non-negative $h$, then
  $\dom\subset \dom_\eps$ and  the interface is $\Gamma=\partial \dom$,
  which is fixed and separates the domain $\dom$ and the thin layer
\[
L_\eps =\set{\bx': \bx' = \bx + th(\bx) \bn(\bx),  0< t< \eps, \bx \in \partial \dom, h(\bx)\neq  0}.
\]
 We call this model   {\it the    thin layer problem.}
The second scenario  is to partition a fixed domain $\dom$ into two subdomains:
$
\dom = \domep \cup  \domen \cup \Gamma_\eps,
$
where $\Gamma_\eps$ is the dividing interface,
which  is assumed as  a perturbation     from
a fixed interface $\Gamma$.  The difference between $\Gamma_\eps$ and
$\Gamma$ can be also described by a function $h$. The detailed definitions of $\domep,\domen,\Gamma_\eps$
will be specified later.
  We call this model   {\it the perturbed interface  problem.}
  In the first problem,
 we attach a thin layer  $L_\eps$ to encircle the fixed domain $\dom$
 and the layer thickness   vanishes as $\eps$ tends to zero.
 The interface there is fixed.
 In the second problem, we  partition  a fixed domain $\dom$
into two subdomains $\dom^{\pm}_{\eps}$ by  a perturbed interface
$\Gamma_\eps$  and the two subdomains 
have comparable size.

All these perturbations can be either deterministic or random,
depending on whether  $h$ is a deterministic function or a random field.
For the latter case,
after $h$ is expended  in random space   by Karhunen-Lo\`{e}ve theorem
$h(\bx,\omega)=\sum h_i(\bx) \phi_i(\omega)$,
or by the Monte Carlo samples
 $h(\bx,\omega) \sim h_i(\bx) $,
the problem   usually can be transformed to a
set of  deterministic  perturbations if the correlation length of
$h$ is not vanishing.  So, we only focus
on the deterministic $h$ here;
the application to the random case
may follow the standard approaches
used   in many literatures such as
\cite{Xiu:2006, HarbLi2013,Castrill2016,Dambrine2016SINUM}.

There is a distinctive class of perturbations of the domain
for the PDE \eqref{eqn:uePDE}:
the so called ``rough boundary/rough domain'',  in which  the spatial scale of the profile
$h$ also depends on $\eps$, for instance,
$\dom$ is perturbed by the form $\eps h_\eps(\bx)=\eps \bar{h}(\bx/\eps)$
for a periodic function $\bar{h}$ (see \cite{Madureira:2017ut} and references therein).
When the boundary condition itself also involves the similar
multiscale feature,    the multiscale finite element method
was applied and analyzed by \cite{MingXu2016}.

To explicitly show the transmission condition
and to introduce our second asymptotic parameter other than
the perturbation size $\eps$,
we take the simplest case of the thin layer problem corresponding to the
first scenario mentioned above.
In this case,  $\dom\subset \dom_\eps$,
 $\Gamma=\partial \dom$ is the interface,
separating the domain $\dom$ and the thin layer
$L_\eps= \dom_\eps \setminus \bar{\dom}$.
Assume that the coefficients $b$ and $c$ vanish and that  $\tensor{a}$ is scalar-valued and
is piecewisely homogeneous   in   $\dom$ and
  $L_\eps$.
Then the corresponding transmission problem takes the form
\begin{equation} \label{Interface-Prob}
\begin{cases}
& -\Delta u_\eps=f   \quad \text{in } {\dom},\\
& -\sigma\Delta  u_\eps =f  \quad \text{in } {L_\eps},\\
& u_\rin=u_\ext, \quad \displaystyle\partial_{\bn} u_\rin= \sigma\partial_{\bn} u_\ext  \quad \text{on } {\Gamma}, \\
& u_\eps=0 \quad \text{on } \partial\dom_\eps,
\end{cases}
\end{equation}
where $\sigma$ is a constant parameter  representing the ratio of conductivity
in two different domains.
$u_\rin$ and $u_\ext$  are the restrictions of the solution $u_\eps$
on two subdomains $D$ and $L_\eps$, respectively.  The similar form of the transmission condition will be specified later
for the general problems. If the material property across the interface
has a significant  difference,  then the value of $\sigma$ can take a very small value or a very large value. The resulted  transmission  problem in
this high-contrast   media
is an important subject  in multiscale analysis and computation.

The elliptic model \eqref{eqn:uePDE}  and the  transmission problem such as
\eqref{Interface-Prob}
originate from many applications such as
diffusion processes, electrostatics, porous media and heat conduction.
One of our motivating examples is
 the diffusion model  of   exciton      in organic semiconductors (\cite{Lin:2014, Guide:2014, Chen:2016}).
 For the discontinuous coefficient model \eqref{Interface-Prob},
a well-known problem is
the electromagnetic model for bodies coated with a dielectric layer
$ L_\eps$ with distinctive material coefficients.
In porous media applications, the permeability of subsurface regions is described as a quantity with high-contrast and multiscale features.

We here mainly concern the asymptotic analysis in terms of the two different
parameters, $\eps$ and $\sigma$, where $\eps$ represents the amplitude of the  geometric perturbation
on the domain or the interface, and $\sigma$ represents the
ratio of different  material coefficients.
In this paper, we shall first consider the asymptotic effect of each  parameter separately and then work on the more complicated  two-parameter expansions.

Many theories and methods have been developed and used to study the
above elliptic problems and the interface problems.
We   review some general methodologies  on the asymptotic study for the solution $u_\eps$
subject to the geometric perturbations. The first idea to handle the   irregular domain $\dom_\eps$ is the {\it domain mapping},
which  is to find a smooth mapping to
change the irregular domain to a fixed  reference  domain.
See the reference \cite{Xiu:2006,Castrill2016,HarbDM2016} for the applications and the analysis of this method.
This method works for any irregular domain as long as a diffeomorphism
can be found regardless it is a small perturbation or not. By applying the diffeomorphism transformation, all geometric information
is transformed into a new differential operator
and a new boundary condition, which
are both more complicated than the original form on irregular domain.
The second method, particularly for the perturbed interface problem, is a generalization of calculus of variation to the geometric setting --- the {\it shape derivative} (\cite{Harbrecht2008,HarbLi2013}).
The method of shape derivatives is  widely used for the sensitivity analysis of the geometry of the boundary
and shape optimization.
Although it is quite easy to obtain the first few order   derivatives,
the calculation  is   very complicated for   the    higher order derivatives.
 The last method, which is also our main tool here, is
 the {\it asymptotic expansion}, which   actually refers to a collection of problem-specific methods
and relies on the correct  use of the ans\"{a}tz
 (\cite{Vial2005,Caloz2006AA,CEG2014,Dambrine2016SINUM}).
In this method,  by using a good  regularity of the solution in the correct (sub)domains,
one can apply certain ans\"{a}tz in the form of  the series expansion to approximate the boundary conditions on the fixed domain.
More details
on the   application of this method to our  problems of concerns
will be reviewed and commented in subsequent sections.

The main motivation of this article   is   to
give a comprehensive study on the (formal) asymptotic expansions
of the solutions to the above various elliptic problems, including the thin layer problem and the
interface problem, up to an arbitrary order in theory.
Specifically, we shall address the following four  problems.

 \begin{enumerate}
 \item[(I)] The first task is that for the elliptic model \eqref{eqn:uePDE}
 with smooth $\tensor{a}$ ,
 we want to have in $\dom$
 \begin{equation}\label{expn}
  u_\eps = u_0 + \eps u_1  + \eps^2 u_2 +  \eps^3 u_3+ \ldots,
 \end{equation}
 in certain sense,
where all terms $u_i$ are independent of $\eps$ explicitly.
Then we want to construct a sequence $\{u^{[n]}\}$ of functions
satisfying the following properties:
(i)   Each $u^{[n]}$ is the solution to a boundary value problem defined only on the fixed domain $\dom$;
(ii)  The error between the restriction of $u_\eps$ to $\dom$ and $u^{[n]}$
is  limited to  the order $\Od(\eps^{n+1})$;
(iii) The numerical computation
(which is not our objective in this paper) of $u^{[n]}$ should be easier than  directly solving the original equation \eqref{eqn:uePDE}.
 Note that $u^{[n]}$ is not simply the partial sum $\sum_{i=0}^n \eps^n\,u_n$, because the latter
may not satisfy a closed boundary value problem.

\item[(II)]
 The second task is to generalize the results in  (I) to the
 thin layer problem \eqref{Interface-Prob} for
the case of the discontinuous   coefficient $a$.

\item[(III)]  The third one is the generalization of (II) to
the   high-contrast material, i.e.,
 $\sigma$, the ratio of  material coefficients across the interface $\Gamma$,
is very large or very small.
We want to derive
 the  two-parameter expansions
when the limits of  both $\eps$  and  $\sigma$ are considered.
We are concerned with the three scaling regimes for $\eps$ and $\sigma$:
 $\eps/\sigma\rightarrow 0$,
  $\eps/\sigma \rightarrow \infty$,
and  $\eps/\sigma\rightarrow c\in(0,~\infty)$.
The final result is  the boundary value problem for each term
in the two-parameter asymptotic  expansions 
$u_\rin(\bx)=\sum_{m,n} u_{m,n}(\bx)\eps^m\mu^n$, where $m,n$ are  integers, and $\mu$ is linked to the ratio of $\eps$ and $\mu$, whose specific form
depends on the asymptotic regimes.
We shall show that the three scalings
will give arise to the Dirichlet,
Neumann
or  Robbin  boundary condition for
 $u_{m,n}$, respectively.

\item[(IV)] The last one is   on the perturbed
  interface problem where
  the interface $\Gamma_\eps$ is not fixed
  as in (II) and (III), but
 is associated with  a perturbed  domain partition
$\dom = \domep \cup  \domen \cup \Gamma_\eps$.
  Meanwhile,  the high-contrast ratio limit is also considered,
and  we derive the two-parameter asymptotic expansions,
  where we find there is no special dependence on
  the scaling of $\eps$ and $\sigma$.

\end {enumerate}

From Section 2 to Section 5, we solve each of these  four problems
in each section.
The techniques we used for (I) and (II) are different from the existing methods.
The   two-parameter  asymptotic expansions
for (III) and (IV)  in this paper are new results.
The main techniques we apply here for all four problems
are the Taylor expansion applied in various contents, which
all requires a good regularity of the underlying function.
For the thin layer problem or the interface problem, where the   solution $u_\eps$
apparently does not posses such smoothness on the interface,
our idea is first to extend each smooth component of the solution $u_\eps$ on each subdomain  onto
some $\eps$-independent domains before applying any asymptotic expansions.
This is achieved by imposing certain Cauchy problems on the interface 
when interpreting the elliptic equation as a time-evolution equation 
in which the normal direction of the interface is the time marching direction.
The second important idea is
to apply the inverse  {Lax-Wendroff} procedure
(\cite{Shu:2010})
to convert the high order derivatives in the normal direction  on the interface to those along
 the tangent directions and  the first order normal derivative, for which the
 original transmission condition on the interface is utilized.

To end this introduction, we review  several existing works
which are closely related to the problems we considered here.
 The work in
\cite{Dambrine2016SINUM}   considered the  thin layer problem   \eqref{Interface-Prob}
with a fixed $\sigma$ as $\eps\to 0$.
The main idea in \cite{Dambrine2016SINUM} is to
 write the differential operator $\Lo$
in terms of local coordinate in the thin layer $L_\eps$,
and apply the ans\"{a}tz  $\Lo = \sum_{n\geq -2} \eps^n \Lo_n$
  to derive  a  system of (infinitely number of) recursive equations
for the expansion of the solution in this dilated layer.
Then with the aid of the transmission condition on the interface $\Gamma$,
the boundary conditions of  these equations in the layer $L_\eps$ are linked to the solutions in the interior (fixed) domain $\dom$.
In \cite{CEG2014},
to assist the construction of local solutions
in the multiscale finite element methods for the    elliptic equations in high-contrast media,  the authors derived asymptotic expansions for the solutions of the elliptic problems with high contrast ratio, i.e., $\sigma$ tends to $0$ or $\infty$. But their analysis is for the fixed domain and interface.

%
%
%
\section{The elliptic problem with   smooth coefficients}
\label{sec:2}
In this section,
we study the equation \eqref{eqn:uePDE} on $\dom_\eps$ by assuming that $\tensor{a}(\bx)$
is sufficiently smooth everywhere and $h(\bx)$ in \eqref{h}
is also sufficiently smooth on $\partial\dom$.   This means that the Taylor expansion for these two functions
are available up to any order.
The signs of $h(\bx)$ can be arbitrary at different $\bx\in\partial\dom$
and the operator $\Lo$ in \eqref{def:Lo} is not limited to the Laplace operator.

Recall that the
perturbed thin layer $L_\eps$ is defined by
\[
L_\eps =\set{\bx': \bx' = \bx + th(\bx) \bn(\bx),  0< t< \eps, \bx \in \partial \dom, h(\bx)\neq  0}.
\]
The condition
$h(\bx) \neq 0$ ensures that $L_\eps$ is also a domain (open set).
Depending on the sign of the function $h$,
we can decompose  the thin layer $L_\eps$  into the interior layer $L_{\eps,\rin}$ and the external layer $L_{\eps,\ext}$:
 $$L_\eps=L_{\eps,\rin}\cup L_{\eps,\ext},$$
 where
\[
L_{\eps,\rin}:=L_\eps\cap\dom
=\set{\bx': \bx' = \bx + t h(\bx) \bn(\bx),  0< t< \eps, \bx \in \partial \dom, h(\bx)<0},
  \]
\[L_{\eps,\ext}:=L_\eps\setminus\dom
=\set{\bx': \bx' = \bx + t h(\bx) \bn(\bx),  0< t< \eps, \bx \in \partial \dom, h(\bx)>0}.
  \]
 $L_{\eps,\rin}\subset \dom$ and $L_{\eps,\ext}\cap \dom=\emptyset$.
Then  $\dom_\eps$ is the interior of $\overline{(\dom\setminus L_{\eps,\rin})\cup L_{\eps,\ext}}.$
Refer to the schematic illustration in Figure \ref{f:shape}.
\begin{figure}[htp]
\includegraphics[width=0.4\textwidth]{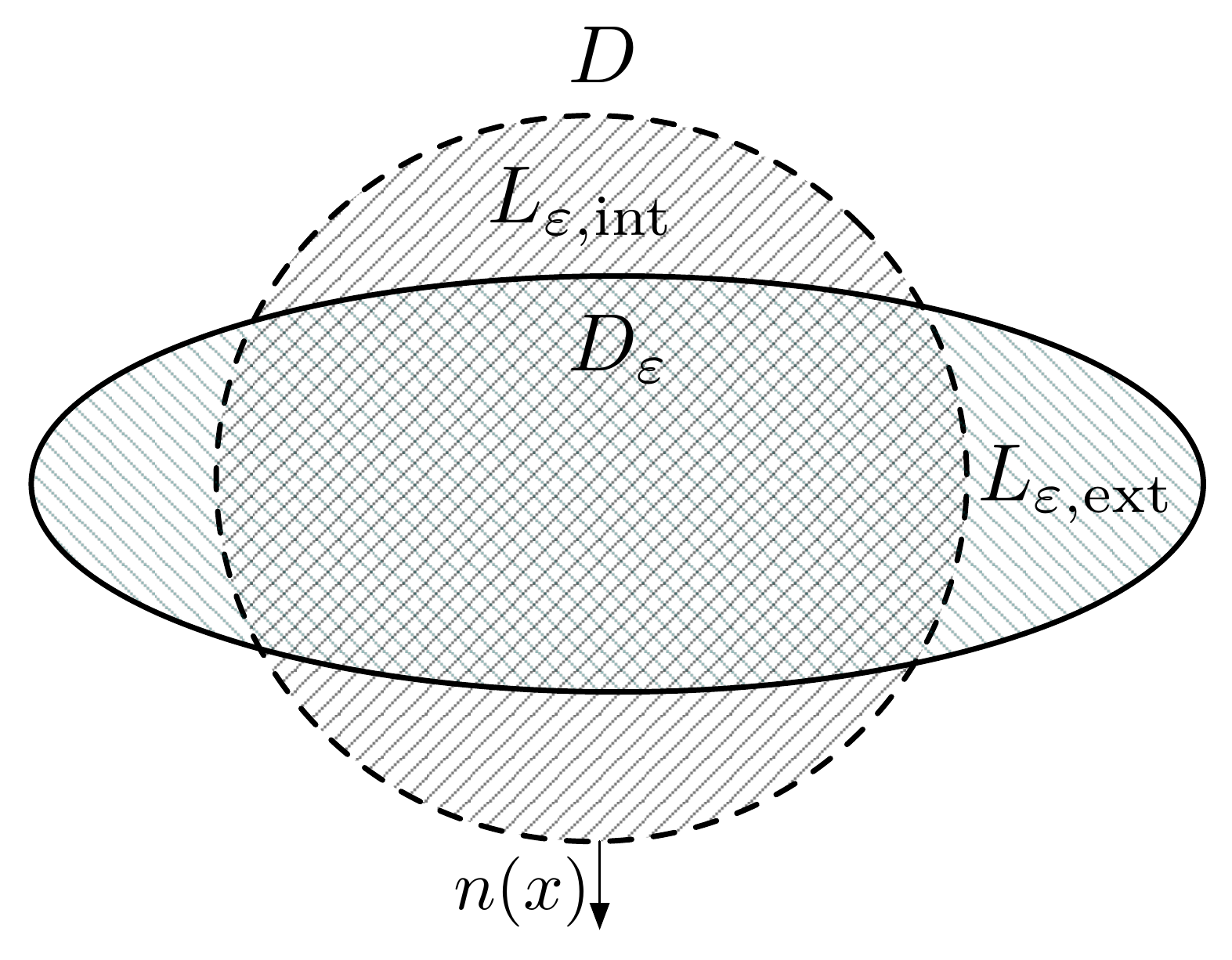}
\caption{Schematic illustration of the domain perturbation. The regular domain $\dom$ is in the ``ball'' shape and the perturbed domain $\dom_\eps$
is in the ``ellipse'' shape.}
\label{f:shape}
\end{figure}

\subsection{Approximate expansions} \label{ssec:AA}
The   problem \eqref{eqn:uePDE} is defined on the $\eps$-dependent  domain $\dom_\eps$.
We extend it to a fixed domain $\dom\cup \dom_{\eps_0}$
and justify this extension
in Section \ref{subsubsec:extension}.
Then in Section \ref{sssec:aymp},
 we use  the Taylor expansion near $\partial\dom$ to derive the
asymptotic expansion  $u_\eps =\sum_{n\geq 0} \eps^n u_n$,  for which
the  inverse Lax-Wendroff procedure is applied to convert the high order normal derivatives
into the first order normal derivative and the tangential derivatives along the boundary $\partial \dom$.

\cite{Dambrine2016SINUM}  already derived   the first three terms, $u_0$, $u_1$ and $u_2$.
But the method we give below seems simpler and
   does not require the dilation technique and any asymptotic form for the
differential operator $\Lo$ used in \cite{Dambrine2016SINUM}.
Actually,  that kind of  singular perturbation suits for the case that the solution itself develops a sharp peak in the thin layer,
such as the traditional boundary layer analysis in fluid mechanics.
However, the problem here does not have this feature
and  the solutions   on $\dom$ and $\dom_\eps$
both behave very normally at the order $\Od(1)$.
We find that the direct expansion for the  boundary condition of $u_\eps$
in an appropriate way is sufficient
to derive the boundary condition  of $u_n$.
To present our main technique, we start with the smooth $\tensor{a}$ case in this section
and then show how to
generalize to  the discontinuous   $\tensor{a}$ in  Section \ref{sec:3}.

\subsubsection{The extension of the solution to the fixed  domain}\label{subsubsec:extension}
Note that $\dom\cup\dom_\eps$ is increasing in $\eps$
since $L_{\eps,\ext}$ always expands as $\eps$ increases.
So it is convenient to make the extension
to the  whole   domain $\dom\cup\dom_{\eps_0}$ since we only consider    $\eps \in (0,~\eps_0]$.
On this fixed domain $\dom\cup\dom_{\eps_0}$, the solution is known on the part  $\overline{\dom_\eps}$;
we thus consider  the difference   $ \triangle_{\eps}$ which
consists of the disjoint  thin layers:
\[
 \triangle_{\eps}:=(\dom\cup\dom_{\eps_0})\setminus\overline{\dom_\eps} =L_{\eps,\rin}\cup N_{\eps}, ~~
\mbox{ where } N_{\eps}:=
 L_{\eps_0,\ext}\setminus \overline{L_{\eps,\ext}}.
\]
Denote the solution extended on $\triangle_{\eps}$ by  $\tu_\eps$,
and assume that  $\tu_\eps$  and $u_\eps$  have the same values and
the same normal derivatives on the common boundary $\partial \dom_\eps$.
Specifically,  $\tu_\eps$ is the unique solution to  the following {\it Cauchy} problem     posed in the thin layers $L_{\eps,\rin}$ and
 $N_{\eps}$:
\begin{equation}\label{eqn:uilPDE}
\begin{cases}
 \Lo \tu_\eps=f \qquad   &\text{in } L_{\eps,\rin}\cup N_{\eps},\\
\tu_\eps=u_\eps=g, \quad \partial_{\bn}  \tu_\eps=\partial_{\bn} u_\eps \quad &\text{on } \partial \dom_\eps,
\end{cases}
\end{equation}
where $u_\eps$, the solution to equation \eqref{eqn:uePDE}, is presumably given,
$\bn$ is the outward normal of $\dom_\eps$ on $\partial\dom_\eps$.
Note that $\partial\dom_\eps$ is  a proper subset of the boundaries of  $L_{\eps,\rin}$ and $ N_{\eps}$.
The problem \eqref{eqn:uilPDE} is actually a Cauchy problem of $\tu_\eps$, not  a boundary-valued elliptic problem,
because the value and the ``velocity'' of $\tu_\eps$ are specified  on $\partial\dom_\eps$ ---
a part of its complete boundary.
The boundary $\partial\dom_\eps$ satisfies the noncharacteristic condition
$\sum_{i,j=1}^da^{ij}n_in_j\neq 0$  trivially since
$\Lo$ is, by assumption, an elliptic operator satisfying \eqref{eqn:ellipcond}. Thus
by the Cauchy-Kovalevskaya theorem  (\cite{Evans:1998}),
the solution on $\partial \dom$ can propagate to the boundary $\partial \triangle_\eps$
and
 the above Cauchy problem \eqref{eqn:uilPDE} is well-posed
 for  sufficiently small  $\eps_0$.
\begin{remark}
The above method of extending the solution to a larger
(and $\eps$-indepedent) domain
 can also preserve the regularity of the solution
and
 helps clarify the rigorous meaning of the
Taylor expansion we shall apply.
This extension idea  by the use of the Cauchy problem
of a time-evolution equation
 will be applied     repeatedly in this paper,
especially for the interface problem so that each smooth component of the solution on each subdomain
may be approximated by the Taylor expansion
along some interface.

\end{remark}

Now it is clear that we can define a function $w_\eps$   piecewisely    on the whole (fixed) domain $\overline{\dom}\cup\overline{\dom_{\eps_0}}=
\overline{\dom_\eps}\cup  \overline{\triangle_\eps}$
as follows:
 \begin{equation}
 w_\eps(\bx) := \begin{cases}
 u_\eps(\bx) &  \mbox{ in } \overline{\dom_\eps},\\
 \tu_\eps(\bx) &  \mbox{ in }  \overline{\triangle_\eps}.
 \end{cases}
 \end{equation}
This definition is justified by the boundary condition in \eqref{eqn:uilPDE} which dictates that $u_\eps$ and $\tu_\eps$ coincide on the common boundary
$\partial \dom_\eps$.
Then $w_\eps$ satisfies the equation on the fixed domain
\begin{equation}\label{eqn:wPDE}
\Lo w_\eps=f \quad \text{in } \dom\cup\dom_{\eps_0},
\end{equation}
and on the $\eps$-dependent  boundary.
\begin{equation}\label{eqn:wBC}
w_\eps=g, \quad  \text{on } \partial \dom_\eps.
\end{equation}
Note that  \eqref{eqn:wBC} does not serve
as a boundary condition to the equation  \eqref{eqn:wPDE}.
  $w_\eps$ is simply
   a   combination of $u_\eps$ from
    the boundary value problem \eqref{eqn:uePDE}  and $\tu_\eps$ from  the Cauchy problem \eqref{eqn:uilPDE}   .
The above argument of extension ensures that $w_\eps$
has the same regularity of $u_\eps$, but   on $\dom\cup\dom_{\eps_0}$.

\subsubsection{Asymptotic expansion on the whole domain}\label{sssec:aymp}
By the above extension,  we can  assume  the  following ans\"{a}tz for $w_\eps$,
\begin{equation}\label{eqn:wans}
w_\eps(\bx)=\sum_{n=0}^\infty \eps^nw_{n}(\bx) \quad \text{for }\bx\in \overline{\dom\cup\dom_{\eps_0}}.
\end{equation}
Plug this ans\"{a}tz into  the equation \eqref{eqn:wPDE}, and
match the terms at the same order of $\eps$,   then we obtain the following equations for   $w_n$
in $\dom\cup\dom_{\eps_0}$:
\begin{equation}\label{eqn:wnPDE}
\Lo w_n=\delta_{0,n}f.
\end{equation}
Here $\delta_{i,j}=1$ if $i=j$ and  $\delta_{i,j}=0$ if $i\neq j$.

For the   condition \eqref{eqn:wBC}, $w_\eps=g$ on $\partial \dom_\eps$,
by noticing the fact that $\bx+\eps h(\bx)\bn(\bx)\in\partial \dom_\eps$
for all $\bx\in\partial \dom$,  we have
\begin{equation}\label{eqn:swnbdry}
w_\eps\bigl(\bx+\eps h(\bx)\bn(\bx)\bigr)=\sum_{n=0}^\infty \eps^nw_n\bigl(\bx+\eps h(\bx)\bn(\bx)\bigr) =g\bigl(\bx+\eps h(\bx)\bn(\bx)\bigr).
\end{equation}
The Taylor expansions in $\eps$ on the right-hand side read
\begin{equation}\label{eqn:Taylor}
w_n\bigl(\bx+\eps h(\bx)\bn(\bx)\bigr)=\sum_{k=0}^\infty \frac{\eps^k\bigl(h(\bx)\bigr)^k}{k!}{\partial_{\bn}^k}w_n(\bx),
\end{equation}
\begin{equation}\label{eqn:gTaylor}
g(\bx+\eps h(\bx)\bn(\bx))=\sum_{k=0}^\infty \frac{\eps^k\bigl(h(\bx)\bigr)^k}{k!}{\partial_{\bn}^k}g(\bx),
\end{equation}
where for any vector field $\bn(\bx)=\bigl(n_1(\bx),\cdots,n_d(\bx)\bigr)$, the $k$-th directional derivative along $\bn$ at $\bx=\bx_0$ is defined by
\[
\partial_{\bn}^k\big|_{\bx=\bx_0}:=\biggr(\sum_{i=1}^d n_i(\bx_0)\partial_{x_i}\bigg|_{\bx=\bx_0}\biggr)^k.
\]
Then \eqref{eqn:swnbdry}, \eqref{eqn:Taylor} and \eqref{eqn:gTaylor}  together lead to
\[
\sum_{n=0}^\infty \eps^n\sum_{k=0}^{\infty}\frac{\eps^{k}\bigl(h(\bx)\bigr)^k}{k!} {\partial_{\bn}^k}w_n(\bx)=\sum_{k=0}^\infty \frac{\eps^k\bigl(h(\bx)\bigr)^k}{k!}{\partial_{\bn}^k}g(\bx),
\]
which, by a change of the indices $m=k+n$, is equivalent to
\[
\sum_{m=0}^{\infty}\eps^{m} \sum_{k=0}^m\frac{\bigl(h(\bx)\bigr)^k}{k!}{\partial_{\bn}^k}w_{m-k}(\bx)=\sum_{m=0}^\infty \frac{\eps^m\bigl(h(\bx)\bigr)^m}{m!}{\partial_{\bn}^m}g(\bx).
\]
Then by matching the terms with the same order of  $\eps$, we obtain that
\[
\sum_{k=0}^m\frac{\bigl(h(\bx)\bigr)^k}{k!}{\partial_{\bn}^k}w_{m-k}(\bx)=\frac{\bigl(h(\bx)\bigr)^m}{m!}{\partial_{\bn}^m}g(\bx),
\]
i.e.,
\begin{equation}\label{eqn:wmbdc}
\begin{cases}
w_0(\bx) = g(\bx), \\
w_{m}(\bx)=\dfrac{\bigl(h(\bx)\bigr)^m}{m!}{\partial_{\bn}^m}g(\bx)
-\displaystyle\sum_{k=1}^m \dfrac{\bigl(h(\bx)\bigr)^k}{k!}{\partial_{\bn}^k}w_{m-k}(\bx), ~~\forall m\geq 1.
\end{cases}
\end{equation}
This provides a recursive expression of  the boundary condition on $\partial \dom$ for
the $m$-th order term $w_m$.

Define   $u_n$ as  the restriction of   $w_n$ to $\dom$. Then
$u_\eps=\sum_{n=0}^\infty \eps^n u_n$. By \eqref{eqn:wnPDE} and \eqref{eqn:wmbdc},  $u_n$ satisfies the following sequence of
boundary value problems on $\dom$ where the boundary conditions on $\partial \dom$
are defined recursively:
\begin{equation}
\label{PDE:u0}
\begin{cases}
&\Lo u_0=f \quad \text{in } \dom,\\
& u_0=g \quad \text{on } \partial \dom,
\end{cases}
\end{equation}
and for $n\geqslant 1$,
\begin{equation}
\label{PDE:un}
\begin{cases}
&\Lo u_n=0 \quad \text{in } \dom,\\
& u_n(\bx)=\dfrac{\bigl(h(\bx)\bigr)^n}{n!}{\partial_{\bn}^n}g(\bx)
-\displaystyle\sum_{k=1}^n \dfrac{\bigl(h(\bx)\bigr)^k}{k!}{\partial_{\bn}^k}u_{n-k}(\bx) \quad \text{on } \partial \dom.
\end{cases}
\end{equation}
In particular, for  $n=1,2,3$, the above boundary conditions  on $\partial \dom$
are
\begin{align}
& u_1(\bx)=h(\bx)\partial_{\bn} g(\bx)-h(\bx)\partial_{\bn} u_0(\bx),\\
& u_2(\bx)=\dfrac{\bigl(h(\bx)\bigr)^2}{2}{\partial_{\bn}^2}g(\bx)-h(\bx)\partial_{\bn} u_1(\bx)
 -\dfrac{\bigl(h(\bx)\bigr)^2}{2}{\partial_{\bn}^2}u_{0}(\bx),
\label{eqn:w20z}\\
& u_3(\bx)=\dfrac{\bigl(h(\bx)\bigr)^3}6\partial_{\bn}^3 g(\bx)-h(\bx)\partial_{\bn} u_2(\bx)-\dfrac{\bigl(h(\bx)\bigr)^2}2\partial_{\bn}^2 u_1(\bx)-\dfrac{\bigl(h(\bx)\bigr)^3}6\partial_{\bn}^3 u_0(\bx).\label{eqn:w30z}
\end{align}
\begin{remark}
Using the shape calculus method, one may also derive a ``shape-Taylor expansion" of $u_\eps$ on any compact set $K\subset \dom\cap\dom_\eps$ (see \cite{Harbrecht2008} and the references therein),
\[
u_\eps(\bx)=u_0(\bx)+\eps \rd[\bU](\bx)+\frac{\eps^2}{2}\rd^2[\bU,\bU](\bx)+\Od(\eps^3),
\]
where $u_0$ is the solution to \eqref{PDE:u0},
$\rd[\bU]$ is the first order shape derivative on the boundary variation $\bU$, which is given by the Dirichlet problem
\[
\begin{cases}
&\Lo\rd[\bU]=0 \quad \text{in } \dom,\\
& \rd[\bU]=\bU\cdot\bn\partial_{\bn}(g-u_0) \quad \text{on } \partial \dom.
\end{cases}
\]
$\rd^2[\bU,\bU']$ is the second order shape derivative, i.e., the ``shape Hessian", on the pair $(\bU,\bU')$ of
boundary variations, which is given by the Dirichlet problem
\[
\begin{cases}
&\Lo\rd^2[\bU,\bU']=0 \quad \text{in } \dom,\\
& \rd^2[\bU,\bU']=\partial_{\bU}\partial_{\bU'}(g-u_0)-\partial_{\bU}\rd[\bU']-\partial_{\bU'}\rd[\bU] \quad \text{on } \partial \dom.
\end{cases}
\]
It is easy to see that when the boundary variation $\bU(\bx)$ is given by $\bU(\bx)=h(\bx)\bn(\bx)$ for $\bx\in\partial\dom$,
then $\rd[\bU]=u_1$ and $\rd^2[\bU,\bU]=2u_2$.
Therefore  the shape calculus method   produces the same result as our method.
\end{remark}

The right-hand side
of the boundary condition   \eqref{PDE:un} for each $u_n$
   involves the normal derivatives
of all lower order terms. The inverse   Lax-Wendroff procedure,
which is used to  construct high order numerical methods such as in \cite{Shu:2010},
 enables us to convert the high order normal derivatives
into the first order normal derivative and the  tangential derivatives
on the boundary $\partial\dom$. See Lemma \ref{lem:convert} below. This conversion  procedure
here seems only  {optional}
in theory, but as we shall show in Section \ref{sec:3},
for piecewisely smooth coefficients, this step is {\it essential} for the use of
transmission conditions on the   interface to {link} the
interior solution   and the exterior solution.

\begin{lemma}\label{lem:convert}
Let $u$ satisfy $\Lo u=f$ where $\Lo$ is the elliptic operator in \eqref{def:Lo}. Then
  all the
 normal derivatives $\partial_{\bn}^k u$ on a smooth surface $\Gamma$ with order $k\geq 2$ can be expressed
  in terms of the boundary $\Gamma$,
the restrictions of  the function $u$ and its normal derivative
$\partial_{\bn}u$ on $\Gamma$, and the coefficient functions $a^{ij}$, $b^i$, $c$, $i,j=1,\cdots, d$.
Therefore for every $k\geqslant 2$, every smooth surface $\Gamma$, every elliptic operator $\Lo$ and
every smooth function $f$, there exists an operator
$$F_{k,\Gamma,\Lo,f}[\cdot,\cdot]$$ acting on a pair of functions defined on $\Gamma$ such that
for any smooth function $u$ satisfying $\Lo u=f$, its $k$-th normal derivative $\partial_{\bn}^k u$
on $\Gamma$ is given by
$F_{k,\Gamma,\Lo,f}[u,\partial_{\bn}u]$.

In addition,
it is easy to see the following properties of the operator $F_{k,\Gamma,\Lo,f}[\cdot,\cdot]$ from the linearity of $\Lo$:
\[
F_{k,\Gamma,\Lo,f}[u,\partial_{\bn}u]+F_{k,\Gamma,\Lo,\varphi}[v,\partial_{\bn}v]=F_{k,\Gamma,\Lo,f+\varphi}[u+v,\partial_{\bn}u+\partial_{\bn}v],
\]
\[
cF_{k,\Gamma,\Lo,f}[u,\partial_{\bn}u]=F_{k,\Gamma,\Lo,cf}[cu,c\partial_{\bn}u], \quad\forall c\in\RR,
\]
where $u$ and $v$ solve $\Lo u=f$ and $\Lo v=\varphi$ respectively.
In particular, taking $c=0$ in the last equality yields $F_{k,\Gamma,\Lo,0}[0,0]=0$.

\end{lemma}
For the proof of this lemma,  refer to Theorem 1 in Section 4.6 of \cite{Evans:1998}.
The crucial assumption for the proof is the noncharacteristic condition of $\Gamma$,
 which is automatically guaranteed by the ellipticity of $\Lo$.
This lemma will be used later multiple times   and the dependency on $\Gamma$ and $\Lo$
in the notation of the mapping $F$ may be dropped out if they are self-explanatory.

With this notation $F$, the boundary condition for $u_n(n\geqslant 1)$ in  \eqref{PDE:un}  can be formally written as
\[
\begin{split}
u_n(\bx)= ~~& \dfrac{\bigl(h(\bx)\bigr)^n}{n!}{\partial_{\bn}^n}g(\bx)
-h(\bx)\partial_{\bn}u_{n-1}(\bx)
\\
&-\displaystyle\sum_{k=2}^n \dfrac{\bigl(h(\bx)\bigr)^k}{k!}F_{k,\partial\dom,\Lo,\delta_{k,n}f}[u_{n-k},\partial_{\bn}u_{n-k}](\bx).
\end{split}
\]

To demonstrate the above theory and show how the conversion of the higher order normal derivatives works, in Appendix \ref{sec:example},
we present two examples in 2D. The first is our motivating example of exciton diffusion
and the second is the Poisson equation.
Furthermore,    in Appendix \ref{sec:example}, we demonstrate  how to generalize our method to
 the Neumann boundary condition and the  reaction-diffusion equation with nonlinear terms.

\subsection{The partial sums}
\label{ssec:partialsum}

We have formally  derived the hierarchic systems of the boundary value problems for the
expansion terms $\set{u_n}$  in Section \ref{ssec:AA}.
We next derive the closed boundary value problems which  the partial sums
approximately satisfy. The procedure is   the same as in \cite{Dambrine2016SINUM}.
Define the partial sums
\[
v^{[n]}(\bx):=\sum_{k=0}^n \eps^ku_{k}(\bx), \quad n\geqslant 0.
\]
On the boundary $\partial \dom$, by using \eqref{PDE:un}, we have
\begin{equation}\label{eqn:vnbc}
\begin{split}
v^{[n]}(\bx)&=\sum_{k=0}^n \eps^ku_{k}(\bx)\\
&=\sum_{k=0}^n\frac{\eps^k\bigl(h(\bx)\bigr)^k}{k!}{\partial_{\bn}^k}g(\bx)-\sum_{k=0}^n\eps^k\sum_{j=1}^k \dfrac{\bigl(h(\bx)\bigr)^j}{j!}{\partial_{\bn}^j}u_{k-j}(\bx) \\
&=\sum_{k=0}^n\frac{\eps^k\bigl(h(\bx)\bigr)^k}{k!}{\partial_{\bn}^k}g(\bx)
-\sum_{j=1}^n\dfrac{\eps^j{\bigl(h(\bx)\bigr)^j}}{j!}{\partial_{\bn}^j}v^{[n-j]}(\bx).
\end{split}
\end{equation}
It is worth pointing out that the system of the boundary value problems for $v^{[n]}$
is defined recursively.  To obtain $v^{[n]}$, one needs to solve the boundary value problems from
$v^{[0]}$ (i.e., $u_0$) up to  $v^{[n-1]}$.  Thus, in total, $(n+1)$  Dirichlet boundary value problems  have to be solved.
However, it is possible to  directly solve one boundary value problem to obtain  the  approximation
with the same order as  $v^{[n]}$ by replacing the $v^{[n-j]}$ terms on the
right-hand side of   \eqref{eqn:vnbc} by $v^{[n]}$.
Then  one  obtains the following  closed boundary value problem,
whose solution is denoted by $u^{[n]}$:
\begin{equation} \label{eqn:UN}
\begin{cases}
&\Lo u^{[n]}=f \quad \text{in } \dom,\\
& \displaystyle\sum_{k=0}^n \dfrac{\eps^k\bigl(h(\bx)\bigr)^k}{k!}{\partial_{\bn}^k}u^{[n]}(\bx)
=\sum_{k=0}^n\frac{\eps^k\bigl(h(\bx)\bigr)^k}{k!}{\partial_{\bn}^k}g(\bx) \quad \text{on } \partial \dom.
\end{cases}
\end{equation}
In particular, the boundary value problems for $u^{[1]}$ and $u^{[2]}$ are
\[
\begin{cases}
&\Lo u^{[1]}=f \quad \text{in } \dom,\\
& u^{[1]}+ \eps h\partial_{\bn} u^{[1]}=g+ \eps h\partial_{\bn} g \quad \text{on } \partial \dom,
\end{cases}
\]
\[
\begin{cases}
&\Lo u^{[2]}=f \quad \text{in } \dom,\\
& u^{[2]}+ \eps h\partial_{\bn} u^{[2]}+\dfrac{\eps^2h^2}2{\partial_{\bn}^2}u^{[2]} =
g+ \eps h\partial_{\bn} g+\dfrac{\eps^2h^2}2{\partial_{\bn}^2}g \quad \text{on } \partial \dom.
\end{cases}
\]


The following theorem gives the approximation error of $v^{[n]}$, whose proof is given in Appendix \ref{sec:proof1}.

\begin{assumption}\label{asmp}
Assume $\dom\subset\dom_\eps\subset\dom_{\eps_0}$ and $\partial\dom\in C^{\infty}$.
Let the operator $\Lo$ given by \eqref{def:Lo} be strictly elliptic in $\dom_{\eps_0}$ and
have the coefficients $a^{ij}$, $b^{i}$, $c$ belong to $C^{\infty}(\overline{\dom_{\eps_0}})$
and $c\geq 0$. Also assume $f,g\in C^{\infty}(\overline{\dom_{\eps_0}})$ and $h\in C^{\infty}(\partial\dom)$.
\end{assumption}
\begin{theorem}\label{thm:error}
Under the Assumption \ref{asmp}, $\forall n, m\geq 0$,
\begin{equation}\label{eqn:vn err}
\bigl\| v^{[n]} -u_\eps \bigr\|_{H^m(\dom)} = \Od (\eps^{n+1}).
\end{equation}
\end{theorem}

The following approximation error of $u^{[n]}$ has been proved in \cite{Dambrine2016SINUM} for $n=0,1,2$,
\begin{equation*}\label{eqn:un err}
\bigl\| u^{[n]} -u_\eps \bigr\|_{H^1(\dom)} = \Od (\eps^{n+1}).
\end{equation*}
Note that although $u^{[n]}$ and $v^{[n]}$ have the same approximation order, there might still be a
considerable difference in the accuracy of their approximation errors due to the effects of the prefactors.
The numerical results in \cite{Dambrine2016SINUM} show that the approximation $u^{[n]}$ produces much less accurate results than $v^{[n]}$ for $n=1,2$.
This can be easily confirmed by  the following simple  one-dimensional example:
\[
\begin{cases}
&  u_\eps''=2 \quad \text{in } \dom_\eps=(0, ~1+\eps),\\
& u_\eps(0)=u_\eps(1+\eps)=0.
\end{cases}
\]
The true solution is $u_\eps(x)=x^2-(1+\eps)x$.
The equation for $u_0$ reads
\[
\begin{cases}
&  u_0''=2 \quad \text{in } \dom=(0,~1),\\
& u_0(0)=u_0(1)=0,
\end{cases}
\]
with the solution $u_0(x)=x^2-x$.
Then the equation for $u_1$ is
\[
\begin{cases}
&  u_1''=0 \quad \text{in } \dom=(0,~1),\\
& u_1(0)=0, \quad u_1(1)=-u_0'(1)=-1.
\end{cases}
\]
So $u_1(x)=-x$, and then the partial sum
$v^{[1]}(x)=u_0(x)+\eps u_1(x)=x^2-x-\eps x$.
Hence
\[
v^{[1]}(x) - u_\eps(x)=0.
\]
The   equation for $u^{[1]}$  is
\[
\begin{cases}
&  \bigl(u^{[1]}\bigr)''=2 \quad \text{in } \dom=(0,~1),\\
& u^{[1]}(0)=0, \quad u^{[1]}(1)+\eps \bigl(u^{[1]}\bigr)'(1)=0.
\end{cases}
\]
We find $u^{[1]}(x)=x^2-\dfrac{1+2\eps}{1+\eps}x$, which is a worse approximation
than $v^{[1]}$ since
$$u^{[1]}(x)-u_\eps(x)=\dfrac{\eps^2}{1+\eps}x =\Od(\eps^2).$$
To attain the zero error as  $v^{[1]}$, one needs
 to proceed to the next order $u^{[2]}$ by solving
\[
\begin{cases}
&  \bigl(u^{[2]}\bigr)''=2 \quad \text{in } \dom=(0,~1),\\
& u^{[2]}(0)=0, \quad u^{[2]}(1)+\eps \bigl(u^{[2]}\bigr)'(1)+\eps^2=0.
\end{cases}
\]
It turns out   $u^{[2]}(x)=v^{[1]}(x)=x^2-x-\eps x$.

\section{The thin layer problem}
\label{sec:3}
Next, we generalize the  above method from the continuous material coefficients $\tensor{a}(\bx)$
to the transmission problem associated with the  piecewisely smooth coefficients.
The Taylor expansion  used in Section \ref{ssec:AA}  is still applicable
since we essentially apply the expansion on each  subdomain where $\tensor{a}$
is smooth.
The  next   step is to
use Lemma
\ref{lem:convert} (the inverse Lax-Wendroff procedure) to convert
the high order normal derivatives on the interface   to the first order normal
derivative and  the tangential derivatives.
This critical step facilitates
 the transmission condition
given on the interface   to  build the connection between the solutions
on each subdomain.

For ease of exposition, we only deal with the outward perturbation where $h(\bx)\geqslant0$
for all $\bx\in \partial \dom$. So $\dom$ is a (proper) subset of $\dom_\eps$ and
the difference $\dom_\eps\setminus\overline{\dom}$ is the thin layer $L_\eps$.
 The transmission condition is thus imposed on
 $$\Gamma:=\set{\bx\in\partial\dom:h(\bx)>0}\subset\partial\dom.$$
 Note that $\overline{\Gamma}=\partial\dom\cap\partial L_\eps$ and $\partial\dom\setminus\Gamma=\set{\bx\in\partial\dom : h(\bx)=0}=\partial\dom\cap\partial\dom_\eps$.
Assume that the second order coefficients $a^{ij}$, $i,j=1,\cdots, d$,
are piecewisely smooth and have jumps only across the transmission interface $\Gamma$.
In addition, the term $f$ on the right-hand side of the equation is also allowed
(but not necessarily) to have jumps on $\Gamma$.
Specifically, we assume for $i,j=1,\cdots, d$,
\[
a^{ij}(\bx)=\begin{cases}
& a^{ij}_{\rin}(\bx)  \quad \text{for } \bx\in{\dom}\\
& a^{ij}_{\ext}(\bx) \quad \text{for } \bx\in L_{\eps_0}
\end{cases}
,
\quad
f(\bx)=\begin{cases}
& f_{\rin}(\bx)  \quad \text{for } \bx\in{\dom}\\
& f_{\ext}(\bx)  \quad \text{for } \bx\in L_{\eps_0}
\end{cases}
,
\]
where $a^{ij}_\rin$ and $f_\rin$ are smooth functions on $\overline{\dom}$ while $a^{ij}_\ext$ and $f_\ext$ smooth on $\overline{L_{\eps_0}}$, and in general,
$a^{ij}_\rin(\bx)\neq a^{ij}_\ext(\bx)$ for $\bx\in\Gamma$.

Write
\[
u(\bx)=\begin{cases}
& u_{\rin}(\bx)  \quad \text{for } \bx\in{\dom},\\
& u_{\ext}(\bx)  \quad \text{for } \bx\in L_\eps,
\end{cases}
\]
then the  transmission problem of our concern takes the form:
\begin{equation}\label{eqn:transmissionPDE}
\begin{cases}
& \Lo u_\rin=f_\rin  \quad \text{in } {\dom},\\
& \Lo  u_\ext=f_\ext  \quad \text{in } {L_\eps},\\
& u_\rin=u_\ext, \quad \displaystyle\sum_{i,j=1}^d a^{ij}_\rin n_i\partial_{x_j} u_\rin= \sum_{i,j=1}^d a^{ij}_\ext n_i\partial_{x_j} u_\ext  \quad \text{on } {\Gamma}, \\
& u_\rin=g\quad \text{on } {\partial\dom_\eps\cap\partial\dom},\\
& u_\ext=g \quad \text{on } {\partial\dom_\eps}\cap\partial L_\eps.
\end{cases}
\end{equation}

\subsection{Asymptotic expansions in $\dom$ and $L_\eps$}
Conceptually, we may first extend the domain of $u_\ext$ to a fixed larger domain $L_{\eps_0}$,
as in Section \ref{subsubsec:extension}, and for simplicity we still use $u_\ext$ for its extension.
Assume the following two ans\"{a}tze for $u_\rin$ and $u_\ext$ respectively:
\begin{align}
&u_{\rin}(\bx)=\displaystyle\sum_{n=0}^\infty\eps^n u_{\rin,n}(\bx) \quad \text{for }\bx\in \dom,\label{eqn:uintans}\\
&u_{\ext}(\bx)=\displaystyle\sum_{n=0}^\infty\eps^n u_{\ext,n}(\bx) \quad \text{for }\bx\in L_{\eps_0}.\label{eqn:uextans}
\end{align}
Plug these ans\"{a}tze into  \eqref{eqn:transmissionPDE}, and
match the terms at the same order of $\eps$,   then we obtain the following equations for $u_{\rin,n}$ and $u_{\ext,n}$,
\[
\Lo u_{\rin,n}=\delta_{0,n}f_\rin ~~ \text{in } \dom,
~~ \mbox{ and } ~\Lo u_{\ext,n}=\delta_{0,n}f_\ext  ~~ \text{in } L_{\eps_0},
\]
and the transmission conditions on $\Gamma$ for $u_{\rin,n}$ and $u_{\ext,n}$,
\begin{align}
& u_{\rin,n}=u_{\ext,n}, \label{eqn:trancondun}\\
& \sum_{i,j=1}^d a^{ij}_\rin n_i\partial_{x_j} u_{\rin,n}= \sum_{i,j=1}^d a^{ij}_\ext n_i\partial_{x_j} u_{\ext,n}.\label{eqn:tranconddnun}
\end{align}
The boundary conditions on $\partial\dom_\eps\cap\partial\dom$ for $u_{\rin,n}$
is $
u_{\rin,n}=\delta_{0,n}g$
and   $u_{\ext,n}$ share the same condition  on $\partial\dom_\eps \cap\partial L_\eps$.

Our goal is to derive  the correct boundary conditions on $\partial\dom$ for $u_{\rin,n}$.
Note that we already have these conditions on $\partial\dom\cap\partial\dom_\eps$,
thus it remains to find the boundary conditions on $\Gamma$ for $u_{\rin,n}$.
To this end, we actually  first derive the boundary conditions on $\Gamma$ for $u_{\ext,n}$,
and then convert $u_{\ext,n}$ to $u_{\rin,n}$ by the transmission conditions
\eqref{eqn:trancondun} and \eqref{eqn:tranconddnun}.

To  work on the exterior solution $u_\ext$, which behaves nicely
in $L_{\eps_0}$,
we apply  the Taylor expansion method used in Section \ref{sssec:aymp} to the ans\"{a}tz \eqref{eqn:uextans} with the boundary condition $u_\ext=g$ on $\partial\dom_\eps\cap \partial L_\eps$.
The obtained result is
the following recursive expression of the boundary conditions on $\Gamma$ for $u_{\ext,n}$:
\begin{equation}\label{eqn:uextbdc}
\begin{cases}
u_{\ext,0} = g, \\
\begin{split}
u_{\ext,n}&=\dfrac{h^n}{n!}{\partial_{\bn}^n}g-\displaystyle\sum_{k=1}^n \dfrac{h^k}{k!}{\partial_{\bn}^k}u_{\ext,n-k}\\
&=\dfrac{h^n}{n!}{\partial_{\bn}^n}g-h\partial_{\bn}u_{\ext,n-1}
\\
&~~~\qquad  -\sum_{k=2}^n \dfrac{h^k}{k!}F_{k,\delta_{k,n}f_\ext}[u_{\ext,n-k},\partial_{\bn}u_{\ext,n-k}],
~~~\forall n\geq 1,
\end{split}
\end{cases}
\end{equation}
where the operator $F_{k,f}$ is the operator $F_{k,\Gamma,\Lo,f}$   introduced in
Lemma \ref{lem:convert} and the subindices $\Gamma$ and $\Lo$ are dropped for simplicity.

To handle the terms $\partial_{\bn}u_{\ext,n-k}$  on the right-hand side of \eqref{eqn:uextbdc},
we need the following lemma, proven in Appendix \ref{sec:proof2}.

\begin{lemma}\label{lem:dnuetodnui}
For any integer $n\geq 0$ and any $\bx \in \Gamma$,
one can uniquely determine the value of the
normal derivative   $\partial_{\bn}u_{\ext,n}(\bx)$ on $\Gamma$
from the information of $u_{\rin,n}$
by using \eqref{eqn:trancondun} and \eqref{eqn:tranconddnun}.
More precisely,  $\partial_{\bn}u_{\ext,n}(\bx)$ for $\bx\in\Gamma$ only depends on
\begin{itemize}
\item
the normal vector $\bn=\bn(\bx)$ and
\item
the value of  $u_{\rin,n}(\bx')$for all $\bx' \in \Gamma$   and
\item    $\partial_{\bn}u_{\rin,n}(\bx)$  and
\item
the second order coefficients
$a^{ij}_\rin(\bx)$ and $a^{ij}_\ext(\bx)$, $i, j=1,\cdots, d$.
\end{itemize}
\end{lemma}

Now  the transmission conditions \eqref{eqn:trancondun} and \eqref{eqn:tranconddnun},
serve the bridge from $u_{\ext,n}$ to $u_{\rin,n}$, with the aid of  Lemma \ref{lem:dnuetodnui}.
Then the calculation following the procedure in  the proof of  Lemma \ref{lem:dnuetodnui} shows
that   \eqref{eqn:uextbdc} leads to the following final results
for the boundary condition of $\set{u_{\rin,n}}$ on $\Gamma$:
\begin{equation}\label{eqn:urinBC}
\begin{cases}
u_{\rin,0} = g, \\
\begin{split}
u_{\rin,n}
&=\dfrac{h^n}{n!}{\partial_{\bn}^n}g-h q_1
-\sum_{k=2}^n \dfrac{h^k}{k!}F_{k,\delta_{k,n}f_\ext}\left[u_{\rin,n-k}, q_k\right],
\forall n\geq 1,
\end{split}
\end{cases}
\end{equation}
where for any $ \bx \in \partial \dom$,
\[ q_k(\bx):=
\frac{Q_\rin(\bn)\partial_{\bn} u_{\rin,n-k}+
\sum_{i,j=1}^d (a^{ij}_\rin-a^{ij}_\ext) n_i\partial_{\btau_j} u_{\rin,n-k}}{Q_\ext(\bn)},
\]
\[
Q_\ext(\bn):=\sum_{i,j=1}^d a^{ij}_\ext n_in_j,  \quad Q_\rin(\bn):=\sum_{i,j=1}^d a^{ij}_\rin n_in_j,
\]

\begin{remark}
Since we have $h=0$ and $u_{\rin,n}=\delta_{0,n}g$ on the boundary $\partial\dom\setminus\Gamma$,
the boundary conditions \eqref{eqn:urinBC}   also holds on $\partial\dom\setminus\Gamma$ and thus on the whole boundary
$\partial\dom$.
\end{remark}

As an illuminating example, let us consider the   elliptic operator $\Lo=-\nabla\cdot\bigl(\sigma(\bx)\nabla\bigr)$ with a discontinuous
 $\sigma(\bx)$, which has been studied in Example \ref{ex:Laplace} when
   $\sigma(\bx)$ is a smooth function.
\begin{example}\label{ex:Laplacetrans}
Set $g=0$ and $\Lo=-\nabla\cdot\bigl(\sigma(\bx)\nabla\bigr)$.
Assume
\begin{equation}\label{eqn:sigma}
\sigma(\bx)=\begin{cases}
& \sigma_{\rin}(\bx)  \quad \text{for } \bx\in{\dom},\\
& \sigma_{\ext}(\bx)  \quad \text{for } \bx\in L_{\eps_0},
\end{cases}
\end{equation}
where $\sigma_\rin$ and $\sigma_\ext$ are smooth functions on $\overline{\dom}$ and $\overline{L_{\eps_0}}$ respectively, and in general,
they are distinct on the common boundary.
To ensure the ellipticity of $\Lo$, we assume that $\sigma_\rin$ and $\sigma_\ext$ are both positive everywhere in their domain.
Then the transmission condition \eqref{eqn:tranconddnun} reads
\[
\sigma_\rin\partial_{\bn}u_{\rin,n}=\sigma_\ext\partial_{\bn}u_{\ext,n}.
\]
Thus we deduce
\begin{equation}\label{eqn:dnuetodnui}
\partial_{\bn}u_{\ext,n}=\frac{\sigma_\rin\partial_{\bn}u_{\rin,n}}{\sigma_{\ext}}.
\end{equation}
Next, we compute explicitly the boundary conditions on $\Gamma$ for the first three orders $u_{\rin,n}$.

Order $n=0$. The boundary condition \eqref{eqn:urinBC}  on $\Gamma$ for $u_{\rin,0}$ is simply
\begin{equation}\label{eqn:ui0bdc}
u_{\rin,0}=0.
\end{equation}

Order $n=1$. The boundary condition \eqref{eqn:uextbdc} on $\Gamma$ for $u_{\ext,1}$ reads
\[
u_{\ext,1}(\bx)=-h(\bx){\partial_{\bn}}u_{\ext,0}(\bx).
\]
Then by \eqref{eqn:trancondun} and \eqref{eqn:dnuetodnui}, we obtain the boundary condition on $\Gamma$ for $u_{\rin,1}$:
\begin{equation}\label{eqn:ui1bdc}
u_{\rin,1}=-\frac{h\sigma_\rin}{\sigma_\ext}{\partial_{\bn}}u_{\rin,0}.
\end{equation}

Order $n=2$. Applying  the   boundary condition \eqref{eqn:u2bdcr} for $u_2$  in Example \ref{ex:Laplace} to  $u_{\ext,2}$ yields
\[
u_{\ext,2}=-h\partial_{\bn}u_{\ext,1}+\frac{h^2}{2}\Bigl(\frac{\partial_{\bn}\sigma_\ext}{\sigma_\ext}+
\kappa\Bigr)\partial_{\bn}u_{\ext,0}+\dfrac{h^2f_\ext}{2\sigma_\ext}.
\]
$\kappa$ is the curvature of $\partial \dom$, defined in Example \ref{ex:Laplace}.
Then substituting \eqref{eqn:trancondun} and \eqref{eqn:dnuetodnui} into the last equation gives
the boundary condition on $\Gamma$ for $u_{\rin,2}$:
\begin{equation}\label{eqn:ui2bdc}
u_{\rin,2}=-\frac{h\sigma_\rin}{\sigma_{\ext}}{\partial_{\bn}}u_{\rin,1}+\frac{h^2\sigma_\rin}{2\sigma_\ext}\Bigl(\frac{\partial_{\bn}\sigma_\ext}{\sigma_\ext}+
\kappa\Bigr)\partial_{\bn}u_{\rin,0}+\dfrac{h^2f_\ext}{2\sigma_\ext}.
\end{equation}
\end{example}

\subsection{The approximate boundary conditions for the partial sums}
Define the partial sums
\begin{equation*}
v^{[n]}(\bx)=
\begin{cases}
&v^{[n]}_{\rin}(\bx):=\displaystyle\sum_{k=0}^n\eps^k u_{\rin,k}(\bx) \quad \text{for }\bx\in \dom,\\
&v^{[n]}_{\ext}(\bx):=\displaystyle\sum_{k=0}^n\eps^k u_{\ext,k}(\bx) \quad \text{for }\bx\in L_{\eps_0}.\\
\end{cases}
\end{equation*}
As in Section \ref{ssec:partialsum},
the goal here is to derive the recursive boundary condition
for the partial sums and to find the closed boundary value problems for the approximations $u^{[n]}$.

To derive the boundary conditions that the partial sums $v^{[n]}_{\rin}$ satisfy, we have two equivalent approaches.
The first
one is to directly derive the boundary conditions for $v^{[n]}_{\rin}$ from  the boundary conditions for $u_{\rin,n}$ which are already obtained above;
the second approach is to apply \eqref{eqn:vnbc} to $v^{[n]}_\ext$ and then transfer to $v^{[n]}_\rin$ via the following transmission conditions
\begin{align*}
& v^{[n]}_{\rin}=v^{[n]}_{\ext}, \\
& \sum_{i,j=1}^d a^{ij}_\rin n_i\partial_{x_j} v^{[n]}_{\rin}= \sum_{i,j=1}^d a^{ij}_\ext n_i\partial_{x_j} v^{[n]}_{\ext},
\end{align*}
which can be easily deduced from \eqref{eqn:trancondun} and \eqref{eqn:tranconddnun}.
Let us  continue to work on  Example \ref{ex:Laplacetrans}
to  illustrate the first approach.

{\it Order $n=1$.}  On the boundary $\partial\dom$, \eqref{eqn:ui0bdc} and \eqref{eqn:ui1bdc} give
\begin{equation}\label{eqn:v[1]bdc}
v^{[1]}_\rin=-\eps\frac{h\sigma_\rin}{\sigma_\ext}{\partial_{\bn}}u_{\rin,0}=-\frac{\eps h\sigma_\rin}{\sigma_\ext} {\partial_{\bn}}v_{\rin}^{[0]}.
\end{equation}
Then we are motivated to introduce the following Robin boundary value problem for $u^{[1]}$:
\begin{equation}\label{eqn:u[1]PDE}
\begin{cases}
& \Lo u^{[1]}=f  \quad \text{in } \dom,\\
& u^{[1]}+\dfrac{\eps h\sigma_\rin}{\sigma_\ext} {\partial_{\bn}}u^{[1]}=0 \quad \text{on } \partial\dom.
\end{cases}
\end{equation}

{\it Order $n=2$}. On the boundary $\partial\dom$, \eqref{eqn:ui2bdc} and \eqref{eqn:v[1]bdc} show
\[
\begin{split}
v^{[2]}_\rin&=-\frac{\eps h\sigma_\rin}{\sigma_\ext} {\partial_{\bn}}v_{\rin}^{[0]}-\frac{\eps^2h\sigma_\rin}{\sigma_{\ext}}{\partial_{\bn}}u_{\rin,1}
+\frac{\eps^2h^2\sigma_\rin}{2\sigma_\ext}\Bigl(\frac{\partial_{\bn}\sigma_\ext}{\sigma_\ext}+
\kappa\Bigr)\partial_{\bn}u_{\rin,0}+\dfrac{\eps^2h^2f_\ext}{2\sigma_\ext}\\
&=-\frac{\eps h\sigma_\rin}{\sigma_\ext} {\partial_{\bn}}v_{\rin}^{[1]}
+\frac{\eps^2h^2\sigma_\rin}{2\sigma_\ext}\Bigl(\frac{\partial_{\bn}\sigma_\ext}{\sigma_\ext}+
\kappa\Bigr)\partial_{\bn}v^{[0]}_{\rin}+\dfrac{\eps^2h^2f_\ext}{2\sigma_\ext},
\end{split}
\]
thus the closed Robin boundary value problem for $u^{[2]}$ can be imposed as:
\begin{equation}\label{eqn:u[2]PDE}
\begin{cases}
& \Lo u^{[2]}=f  \quad \text{in } \dom,\\
& u^{[2]}+\dfrac{\eps h\sigma_\rin}{\sigma_\ext} {\partial_{\bn}}u^{[2]}
-\dfrac{\eps^2h^2\sigma_\rin}{2\sigma_\ext}\Bigl(\dfrac{\partial_{\bn}\sigma_\ext}{\sigma_\ext}+
\kappa\Bigr)\partial_{\bn}u^{[2]}=\dfrac{\eps^2h^2f_\ext}{2\sigma_\ext} \quad \text{on } \partial\dom.
\end{cases}
\end{equation}

To compare with the results derived  in \cite{Dambrine2016SINUM}
where the coefficient $\sigma$ is piecewise constant, we set  $\sigma_\rin=\sigma_0$ and $\sigma_\ext=1$
and  $h(\bx)>0$ for all $\bx\in\partial\dom$.
Then \eqref{eqn:u[1]PDE} becomes
\[
\begin{cases}
& \Lo u^{[1]}=f  \quad \text{in } \dom,\\
& u^{[1]}+{\eps h\sigma_0} {\partial_{\bn}}u^{[1]}=0 \quad \text{on } \partial\dom,
\end{cases}
\]
which is the same as that in  \cite{Dambrine2016SINUM}.
The  equation    \eqref{eqn:u[2]PDE} becomes
\[
\begin{cases}
& \Lo u^{[2]}=f  \quad \text{in } \dom,\\
& u^{[2]}+{\eps h\sigma_0}\Bigl(1-\dfrac{\eps \kappa h}{2}\Bigr) {\partial_{\bn}}u^{[2]}
=\dfrac{\eps^2h^2f}{2} \quad \text{on } \partial\dom.
\end{cases}
\]
Multiplying the boundary condition for $u^{[2]}$ by $(1+\frac{\eps \kappa h}{2})$ yields
\[
\Bigl(1+\frac{\eps \kappa h}{2}\Bigr)u^{[2]}+{\eps h\sigma_0}\Bigl(1-\dfrac{\eps^2 \kappa^2 h^2}{2}\Bigr) {\partial_{\bn}}u^{[2]}
=\Bigl(1+\frac{\eps \kappa h}{2}\Bigr)\dfrac{\eps^2h^2f}{2},
\]
that is,
\[
\Bigl(1+\frac{\eps \kappa h}{2}\Bigr)u^{[2]}+{\eps h\sigma_0} {\partial_{\bn}}u^{[2]}
=\dfrac{\eps^2h^2f}{2}+ \Od(\eps^3).
\]
By neglecting the third order term $\Od(\eps^3)$,
we have
the same equation in \cite{Dambrine2016SINUM}  for $u^{[2]}$.

\section{The thin layer problem with high-contrast ratio}
\label{sec:4}
From this section, we take into account of the contrast ratio
parameter $\sigma$ together with the geometric perturbation
parameter $\eps$.
This section   considers the following transmission problem on $\dom_\eps$:
\begin{equation}\label{eqn:transPDE2para}
\begin{cases}
& -\Delta u_\rin=f_\rin  \qquad \text{in } {\dom},\\
& -\sigma\Delta  u_\ext=f_\ext \quad \text{in } {L_\eps},\\
& u_\rin=u_\ext \quad \displaystyle\partial_{\bn} u_\rin= \sigma\partial_{\bn} u_\ext,  \quad \text{on } {\Gamma}, \\
& u_\rin=g \quad \text{on } \partial\dom_\eps\cap\partial\dom,\\
& u_\ext=g \quad \text{on } {\partial\dom_\eps\cap \partial L_\eps},
\end{cases}
\end{equation}
where $\sigma$ is  a positive constant.
The geometry of the domains are exactly the same as in Section \ref{sec:3}, i.e.,
$\dom\subset\dom_\eps$ and $L_\eps = \dom_\eps\setminus \overline{\dom}$.
$\Gamma$ is the interface separating two materials with different conductivity.
A large  $\sigma$ means a large conductivity in the thin layer $L_\eps$
and a small   $\sigma$ means a (relatively) large conductivity in the interior $\dom$.

We want to investigate the limiting behavior, as well as the asymptotic expansions,  of the interior solution $u_\rin$
as $\eps\rightarrow 0$ and $\sigma\rightarrow 0$ or $ \sigma\rightarrow \infty$.   Before we present the abstract analysis,
let us first
heuristically
show how three scaling regimens can appear by  considering  a simple 1D example.
\begin{example}\label{ex:1d2para}
Let $\dom=(0,~1)$, $\dom_\eps=(-\eps h_0,~1+\eps h_1)$ with two numbers $h_0, h_1\geq 0$, and take $f_\rin=f_\ext=-2$ and $g=1$.
Then it is easy to find the interior solution is
\begin{equation}\label{eqn:1dsol}
\begin{split}
u_\rin(x)=&x^2-Ax+B-
\dfrac{h_0h_1\eps}{\sigma}C,
\end{split}
\end{equation}
and the exterior  solution is
\begin{equation}\label{eqn:1dextsol}
u_\ext(x)=
\begin{cases}
\dfrac{x^2-Ax}\sigma+ B-\dfrac{h_0h_1\eps}\sigma C,
& -h_0\eps \leq x \leq 0, \vspace{2\jot}\\
\dfrac{x^2-Ax+A-1}\sigma+(1-A+B)-\dfrac{h_0h_1\eps}{\sigma}C,
& 1\leq x\leq 1+h_1\eps,
\end{cases}
\end{equation}
where
\[
A=\dfrac{\sigma+2h_1\eps+(h_1^2-h_0^2)\eps^2}{\sigma+(h_0+h_1)\eps},\quad
B=\dfrac{\sigma+h_1\eps-h_0^2\eps^2}{\sigma+(h_0+h_1)\eps}, \quad C=\dfrac{2\eps+(h_0+h_1)\eps^2}{\sigma+(h_0+h_1)\eps}.
\]
\end{example}

The limiting behavior of the interior solution  \eqref{eqn:1dsol} and exterior solution \eqref{eqn:1dextsol}  for this   example is different in the following three cases \begin{itemize}
\item[(i)] $\eps/\sigma\rightarrow 0$,
\item[(ii)] $\sigma/\eps\rightarrow 0$,
\item[(iii)] $\eps/\sigma\rightarrow c$, where $c\in(0,~\infty)$.
\end{itemize}

In Case (i), as $\eps$ and $\mu:=\eps/\sigma$ tend to 0,
we have the interior solution \eqref{eqn:1dsol}
$
u_\rin(x)
\rightarrow x^2-x+1 \sim \Od(1),
$
 and the exterior solution $
 u_\ext(x)
\sim \Od(\sigma^{-1})+\Od(1)+\Od(\mu)=\Od(\mu/\eps)+\Od(1).
 $

In Case (ii), introduce $\lambda:=\sigma/\eps$, then both $\eps$ and $\lambda$ go to 0,
If $h_0h_1>0$, i.e.,  the domain perturbation is applied to the whole boundary $\partial\dom$, then
$u_\rin$ is at the order $\Od(\lambda^{-1})$; otherwise, one has $h_0=0$ or $h_1=0$, and so
$
u_\rin(x) \sim \Od(1)
$.
In both circumstances, $u_{\ext}$ is at the order $\Od(\sigma^{-1})=\Od(\eps^{-1}\lambda^{-1})$.

In Case (iii), as $\eps\rightarrow0$ and $\eps/\sigma\rightarrow c$,
$u_\rin(x)\sim \Od(1) $ and $u_{\ext}$ is at the order $\Od(\sigma^{-1})=\Od(\eps^{-1})$.

For general problems, the scalings of the magnitudes of $u_\rin$ and $u_\ext$ behave  exactly the same
as in the above  example. In the next,
we develop the two-parameter  asymptotic analysis for  the general transmission problem \eqref{eqn:transPDE2para} by  discussing the above three cases.
The results we obtained below  are written recursively up to any order in an abstract way.
The  readers can find explicit boundary conditions and solvability conditions for some lower order terms for
each case in Appendix \ref{sec:case}.

\subsection{Case (i): $\eps/\sigma
\rightarrow 0$, $\eps\rightarrow 0$.}

We   now treat $\eps$ and $$\mu=\eps/\sigma$$ as independent small parameters.
Introduce the rescaled exterior solution
$\tu_\ext=\eps u_\ext,$
then rewrite the original equation \eqref{eqn:transPDE2para} in terms of $u_\rin$ and $\tu_\ext$:
\begin{equation}\label{eqn:transPDE2parac1}
\begin{cases}
& -\Delta u_\rin=f_\rin  \quad \text{in } {\dom},\\
& -\Delta  \tu_\ext=\mu f_\ext  \quad \text{in } {L_\eps},\\
& \tu_\ext=\eps u_\rin, \quad \displaystyle \partial_{\bn} \tu_\ext= \mu\partial_{\bn} u_\rin  \quad \text{on } {\Gamma}, \\
& u_\rin=g \quad \text{on } \partial\dom_\eps\cap\partial\dom,\\
& \tu_\ext=\eps g \quad \text{on } {\partial\dom_\eps\cap \partial L_\eps}.
\end{cases}
\end{equation}
Assume
$u_\rin$ and $\tu_\ext$  have double asymptotic expansions
\[
\begin{split}
u_\rin(\bx)&=\sum_{m,n=0}^\infty u_{\rin,m,n}(\bx)\eps^m\mu^n,  \quad \bx\in\dom,
\\
\tu_\ext(\bx)&=\sum_{m,n=0}^\infty \tu_{\ext,m,n}(\bx)\eps^m\mu^n, \quad \bx\in L_{\eps_0}.
\end{split}
\]
After substituting these into \eqref{eqn:transPDE2parac1} and equating terms of each pair
of powers of $\eps$ and $\mu$, we get the following results:
\begin{align*}
& -\Delta u_{\rin,m,n}=\delta_{0,m}\delta_{0,n}f_\rin \quad \text{in } {\dom},\\
& -\Delta  \tu_{\ext,m,n}=\delta_{0,m}\delta_{1,n}f_\ext  \quad \text{in } {L_{\eps_0}},\\
& \tu_{\ext,m,n}=\begin{cases}
u_{\rin,m-1,n}, \quad & m\geqslant 1\\
 0,\quad &  m= 0\\
\end{cases} \quad \text{on }\Gamma,\\
& \partial_{\bn} \tu_{\ext,m,n}=\begin{cases}
\partial_{\bn} u_{\rin,m,n-1}, \quad & n\geqslant 1\\
 0,\quad &  n= 0\\
\end{cases} \quad \text{on }\Gamma,\\
& u_{\rin,m,n}=\delta_{0,m}\delta_{0,n}g \quad \text{on } {\partial\dom_\eps\cap\partial\dom}.
\end{align*}
For the boundary condition $\tu_\ext=\eps g$ on $\partial\dom_\eps\cap\partial L_\eps$, applying the Taylor expansion method
as in Section \ref{sec:2} and Section \ref{sec:3} yields the following recursive boundary conditions on $\Gamma$ for $\tu_{\ext,m,n}$:
\begin{equation}\label{eqn:tumn}
\begin{cases}
\tu_{\ext,0,n} = 0, \\
\begin{split}
\tu_{\ext,m,n}=&\delta_{0,n}\dfrac{h^{m-1}}{(m-1)!}\partial_{\bn}^{m-1} g-\displaystyle\sum_{k=1}^m \dfrac{h^k}{k!}{\partial_{\bn}^k}\tu_{\ext,m-k,n}\\
=&\delta_{0,n}\dfrac{h^{m-1}}{(m-1)!}\partial_{\bn}^{m-1} g-h\partial_{\bn}\tu_{\ext,m-1,n}
\\
&-\sum_{k=2}^m\dfrac{h^k}{k!}F_{k,\delta_{k,m}\delta_{0,n}f_\ext}
[\tu_{\ext,m-k,n},\partial_{\bn}\tu_{\ext,m-k,n}],
\quad\forall m\geq 1.
\end{split}
\end{cases}
\end{equation}

Next, we transform these boundary conditions on $\Gamma$ for $\tu_{\ext,m,n}$ into those for $u_{\rin,m,n}$.
One has on the interface $\Gamma$
\[
\begin{split}
u_{\rin,m,n}=&\tu_{\ext,m+1,n}=\delta_{0,n}\dfrac{h^{m}}{m!}\partial_{\bn}^{m} g-h\partial_{\bn}\tu_{\ext,m,n}
\\
&~~~\qquad -\sum_{k=2}^{m+1}\dfrac{h^k}{k!}F_{k,\delta_{k,m+1}\delta_{1,n}f_\ext}
[\tu_{\ext,m+1-k,n},\partial_{\bn}\tu_{\ext,m+1-k,n}].
\end{split}
\]
Thus for $m=0$, we have on $\Gamma$
\[
\begin{split}
u_{\rin,0,n}=&\delta_{0,n} g-h\partial_{\bn}\tu_{\ext,0,n}
=\begin{cases}
g, & n=0,\\
-h\partial_{\bn}u_{\rin,0,n-1}, & n\geq 1;
\end{cases}
\end{split}
\]
and for $m\geq 1$ and $n=0$, on $\Gamma$
\[
\begin{split}
u_{\rin,m,0}=
&\dfrac{h^{m}}{m!}\partial_{\bn}^{m} g
-\displaystyle\sum_{k=2}^m\dfrac{h^k}{k!}F_{k,0}
[u_{\rin,m-k,0},0]
-\dfrac{h^{m+1}}{(m+1)!}F_{m+1,0}[0,0]\\
=&\dfrac{h^{m}}{m!}\partial_{\bn}^{m} g
-\displaystyle\sum_{k=2}^m\dfrac{h^k}{k!}F_{k,0}
[u_{\rin,m-k,0},0].
\end{split}
\]
Note that here we used the trivial fact $F_{m+1,0}[0,0]=0$ by definition.

For $m, n\geq 1$, on $\Gamma$
\[
\begin{split}
u_{\rin,m,n}
=&-h\partial_{\bn}u_{\rin,m,n-1}-\displaystyle\sum_{k=2}^m\dfrac{h^k}{k!}F_{k,0}
[u_{\rin,m-k,n},\partial_{\bn}u_{\rin,m+1-k,n-1}]\\
&-\dfrac{h^{m+1}}{(m+1)!}F_{m+1,\delta_{1,n}f_\ext}[0,\partial_{\bn}u_{\rin,0,n-1}].
\end{split}
\]

\subsection{Case (ii):   $\sigma/\eps\rightarrow0$, $\eps\rightarrow 0$.}

Now both  $\eps$ and $$\lambda:=\sigma/\eps$$ are   small parameters.
Introduce $\hu_\ext=\sigma u_\ext=\lambda\eps u_\ext.$
Then \eqref{eqn:transPDE2para} becomes
\[
\begin{cases}
& -\Delta u_\rin=f_\rin \quad \text{in } {\dom},\\
& -\Delta  \hu_\ext=f_\ext  \quad \text{in } {L_\eps},\\
& \hu_\ext=\eps\lambda u_\rin \quad \displaystyle \partial_{\bn} \hu_\ext= \partial_{\bn} u_\rin,  \quad \text{on } {\Gamma}, \\
& u_\rin=g \quad \text{on } \partial\dom_\eps\cap\partial\dom,\\
& \hu_\ext=\eps\lambda g \quad \text{on } {\partial\dom_\eps\cap \partial L_\eps}.
\end{cases}
\]
We have to further  study  two subcases and treat them  separately.

\subsubsection{Case (ii)$_1$: $\Gamma\neq\partial\dom$, or $\partial\dom\cap\partial\dom_\eps\neq\emptyset$}
This means the domain perturbation is only applied to a proper subset $\Gamma$ of the boundary $\partial\dom$.

Assume the double asymptotic expansions
\[\begin{split}
u_\rin(\bx)&=\sum_{m,n=0}^\infty u_{\rin,m,n}(\bx)\eps^m\lambda^n,  \quad \bx\in\dom,
\\
\hu_\ext(\bx)&=\sum_{m,n=0}^\infty \hu_{\ext,m,n}(\bx)\eps^m\lambda^n, \quad \bx\in L_{\eps_0}.
\end{split}
\]
Substituting these into \eqref{eqn:transPDE2para} and equating terms of each pair
of powers of $\eps$ and $\lambda$, we find that
\begin{align*}
& -\Delta u_{\rin,m,n}=\delta_{0,m}\delta_{0,n}f_\rin \quad \text{in } {\dom},\\
& -\Delta  \hu_{\ext,m,n}=\delta_{0,m}\delta_{0,n}f_\ext \quad \text{in } {L_{\eps_0}},\\
& \hu_{\ext,m,n}=\begin{cases}
 u_{\rin,m-1,n-1}, \quad & m, n\geqslant 1\\
 0,\quad & \text{otherwise}\\
\end{cases} \quad \text{on }\Gamma,\\
& \displaystyle\partial_{\bn} \hu_{\ext,m,n}=
\partial_{\bn} u_{\rin,m,n} \quad \text{on }\Gamma, \\
& u_{\rin,m,n}=\delta_{0,m}\delta_{0,n}g \quad \text{on } {\partial\dom_\eps\cap\partial\dom}.
\end{align*}
Applying the Taylor expansion method to the boundary condition $\hu_\ext=\eps\lambda g$ on $\partial\dom_\eps\cap\partial L_\eps$,
we obtain the following recursive boundary conditions on $\Gamma$ for $\hu_{\ext,m,n}$:
\begin{equation}\label{eqn:humn}
\begin{cases}
\hu_{\ext,0,n} = 0, \\
\begin{split}
\hu_{\ext,m,n}=&\delta_{1,n}\dfrac{h^{m-1}}{(m-1)!}\partial_{\bn}^{m-1} g-\displaystyle\sum_{k=1}^m \dfrac{h^k}{k!}{\partial_{\bn}^k}\hu_{\ext,m-k,n}\\
=&\delta_{1,n}\dfrac{h^{m-1}}{(m-1)!}\partial_{\bn}^{m-1} g-h\partial_{\bn}\hu_{\ext,m-1,n}
\\
&
-\sum_{k=2}^m\dfrac{h^k}{k!}F_{k,\delta_{k,m}\delta_{0,n}f_\ext}
[\hu_{\ext,m-k,n},\partial_{\bn}\hu_{\ext,m-k,n}],
\quad\forall m\geq 1.
\end{split}
\end{cases}
\end{equation}

Next, we convert these boundary conditions for $\hu_{\ext,m,n}$ into those for $u_{\rin,m,n}$.
It turns out that   the {\it Neumann  boundary condition} on $\Gamma$ appears in this case.
For $m=0$,  on $\Gamma$
\[
u_{\rin,0,n}=\hu_{\ext,1,n+1}=\delta_{0,n}g-h\partial_{\bn}\hu_{\ext,0,n+1}=\delta_{0,n}g-h\partial_{\bn}u_{\rin,0,n+1},
\]
thus we get for $n\geqslant 1$,
\[
\partial_{\bn}u_{\rin,0,n}=\frac{\delta_{1,n}g}h-\frac1h u_{\rin,0,n-1};
\]
moreover, we have for $m=n=0$, on $\Gamma$
\[
\partial_{\bn}u_{\rin,0,0}=\partial_{\bn}\hu_{\ext,0,0}=-\frac1h\hu_{\ext,1,0}=0.
\]
For $m\geqslant 1$ and $n=0$ , one has on $\Gamma$,
\[
\begin{split}
&\partial_{\bn}u_{\rin,m,0}=\partial_{\bn}\hu_{\ext,m,0}\\
=&-\frac1h\hu_{\ext,m+1,0}-\sum_{k=2}^{m+1}\dfrac{h^{k-1}}{k!}F_{k,\delta_{k,m+1}f_\ext}
[\hu_{\ext,m+1-k,0},\partial_{\bn}\hu_{\ext,m+1-k,0}]\\
=&-\sum_{k=2}^{m+1}\dfrac{h^{k-1}}{k!}F_{k,\delta_{k,m+1}f_\ext}
[0,\partial_{\bn}u_{\rin,m+1-k,0}];
\end{split}
\]
and for $m,n\geqslant 1$, on $\Gamma$
\[
\begin{split}
&\partial_{\bn}u_{\rin,m,n}=\partial_{\bn}\hu_{\ext,m,n}\\
=&\delta_{1,n}\dfrac{h^{m-1}}{m!}\partial_{\bn}^{m} g-\frac1h{\hu_{\ext,m+1,n}}-\sum_{k=2}^{m+1}\dfrac{h^{k-1}}{k!}F_{k,0}[\hu_{\ext,m+1-k,n},\partial_{\bn}\hu_{\ext,m+1-k,n}]\\
=&\delta_{1,n}\dfrac{h^{m-1}}{m!}\partial_{\bn}^{m} g-\frac1h u_{\rin,m,n-1}-\sum_{k=2}^{m}\dfrac{h^{k-1}}{k!}F_{k,0}[u_{\rin,m-k,n-1},\partial_{\bn}u_{\rin,m+1-k,n}]\\
&-\dfrac{h^{m}}{(m+1)!}F_{m+1,0}[0,\partial_{\bn}u_{\rin,0,n}].
\end{split}
\]

Note that the boundary conditions  on $\partial\dom$ for $u_{\rin,m,n}$ are
the mixture  of the Neumann
conditions on $\Gamma$ and the  Dirichlet conditions
$u_{\rin,m,n}=\delta_{0,m}\delta_{0,n}g$
on $\partial\dom\setminus\Gamma$.

 \subsubsection{Case (ii)$_2$: $\Gamma=\partial\dom$, or $\partial\dom\cap\partial\dom_\eps=\emptyset$}
 \label{subsubsec:Case 2.2}
In this case, the domain perturbation is applied to  the whole boundary $\partial\dom$.
It turns out that $u_{\rin}$ is at the order $\Od(\lambda^{-1})$.
So we assume
\[
u_\rin(\bx)=\sum_{m=0}^\infty\sum_{n=-1}^\infty u_{\rin,m,n}(\bx)\eps^m\lambda^n,  \quad \bx\in\dom.
\]
Consequently, the transmission conditions on $\Gamma=\partial\dom$ become
\begin{align*}
& \hu_{\ext,m,n}=\begin{cases}
 u_{\rin,m-1,n-1}, \quad & m\geqslant 1 \\
 0,\quad & m=0
\end{cases} \quad \text{on }\partial\dom,\\
& \displaystyle\partial_{\bn} u_{\rin,m,n}=\begin{cases}
\partial_{\bn} \hu_{\ext,m,n}, \quad & n\geqslant 0\\
 0,\quad & n=-1
\end{cases} \quad \text{on }\partial\dom.
\end{align*}
In addition, \eqref{eqn:humn} still holds.
We already have
$
\partial_{\bn}u_{\rin,m,-1}=0;
$ on $\partial\dom$,
and for $n\geq 0$, on $\partial\dom$
\[
\begin{split}
\partial_{\bn}u_{\rin,m,n}=&\partial_{\bn}\hu_{\ext,m,n}\\
=&\delta_{1,n}\dfrac{h^{m-1}}{m!}\partial_{\bn}^m g-\frac1h \hu_{\ext,m+1,n}\\
&-\sum_{k=2}^{m+1}\dfrac{h^{k-1}}{k!}F_{k,\delta_{k,m+1}\delta_{0,n}f_\ext}
[\hu_{\ext,m+1-k,n},\partial_{\bn}\hu_{\ext,m+1-k,n}].
\end{split}
\]
Thus  on $\partial\dom$, one has for $m=0$, $n\geq 0$,
\[
\partial_{\bn}u_{\rin,0,n}=\dfrac{\delta_{1,n}g}h-\frac1h \hu_{\ext,1,n}=\dfrac{\delta_{1,n}g}h-\frac1h u_{\rin,0,n-1};
\]
and for $m\geq 1$, $n\geq 0$,
\[
\begin{split}
\partial_{\bn}u_{\rin,m,n}
=&\delta_{1,n}\dfrac{h^{m-1}}{m!}\partial_{\bn}^m g-\frac1h u_{\rin,m,n-1}-\dfrac{h^{m}}{(m+1)!}F_{m+1,\delta_{0,n}f_\ext}
[0,\partial_{\bn}u_{\rin,0,n}]\\
&-\sum_{k=2}^{m}\dfrac{h^{k-1}}{k!}F_{k,0}
[u_{\rin,m-k,n-1},\partial_{\bn}u_{\rin,m+1-k,n}].
\end{split}
\]

The above Neumann boundary value problems for $u_{\rin,m,n}$
 are   not well-posed, since
the solution to the Poisson equation with pure Neumann boundary condition
\[
\begin{cases}
-\Delta u=f, & \text{in }\dom,\\
\partial_{\bn} u=g, \quad  &\text{on }\partial\dom,
\end{cases}
\]
can only be determined up to constant.
However, note that a necessary condition for the existence of a solution to the Neumann problem is
\[
\int_{\partial\dom} g=-\int_\dom f.
\]
Applying this solvability condition to the Neumann problem for $u_{\rin,m,n+1}$ leads to
an additional boundary integral condition for $u_{\rin,m,n}$.
Specifically,  the following solvability conditions  can uniquely determine $u_{\rin,m,n}$:
\[
\int_{\partial\dom}\dfrac{u_{\rin,0,n}}h
=\int_{\dom}\delta_{-1,n}f_\rin+\int_{\partial\dom}\frac{\delta_{0,n}g}h,
~~~\qquad n\geq 0,
\]
and  for $m\geq 1$, $n\geq 0$,
\[
\begin{split}
\int_{\partial\dom}\dfrac{u_{\rin,m,n}}h
=
 &~\quad \int_{\partial\dom} \delta_{0,n}\dfrac{h^{m-1}}{m!}\partial_{\bn}^m g
\\
&-\sum_{k=2}^{m}\int_{\partial\dom}\dfrac{h^{k-1}}{k!}F_{k,0}
[u_{\rin,m-k,n},\partial_{\bn}u_{\rin,m+1-k,n+1}]\\
&-\int_{\partial\dom}\dfrac{h^{m}}{(m+1)!}F_{m+1,\delta_{-1,n}f_\ext}
[0,\partial_{\bn}u_{\rin,0,n+1}].
\end{split}
\]

\subsection{Case (iii): $\eps/\sigma\rightarrow c\in (0,~\infty)$, $\eps\rightarrow 0$.}
For this case, we introduce the small parameter
\[
\theta:=\frac\eps\sigma-c,
\]
and also rescale the exterior solution  $\tu_\ext=\eps u_\ext$ as in Case (i).
Plugging the ans\"{a}tz
\[\begin{split}
u_\rin(\bx)&=\sum_{m,n=0}^\infty u_{\rin,m,n}(\bx)\eps^m\theta^n  \quad \bx\in\dom,
\\
\tu_\ext(\bx)&=\sum_{m,n=0}^\infty \tu_{\ext,m,n}(\bx)\eps^m\theta^n \quad \bx\in L_{\eps_0}.
\end{split}
\]
into \eqref{eqn:transPDE2parac1} yields the following
\begin{align}
& -\Delta u_{\rin,m,n}=\delta_{0,m}\delta_{0,n}f_\rin  \quad \text{in } {\dom},\\
& -\Delta  \tu_{\ext,m,n}=\delta_{0,m}\delta_{0,n}cf_\ext+\delta_{0,m}\delta_{1,n}f_\ext  \quad \text{in } {L_{\eps_0}},\\
& \tu_{\ext,m,n}=\begin{cases}
u_{\rin,m-1,n}, \quad & m\geqslant 1\\
 0,\quad &  m= 0\\
\end{cases} \quad \text{on }\Gamma,\\
& \partial_{\bn} \tu_{\ext,m,n}=\begin{cases}
\partial_{\bn} u_{\rin,m,n-1}+c\partial_{\bn} u_{\rin,m,n}, \quad & n\geqslant 1\\
c\partial_{\bn} u_{\rin,m,n},\quad &  n= 0
\end{cases} \quad \text{on }\Gamma,\\
& u_{\rin,m,n}=\delta_{0,m}\delta_{0,n}g \quad \text{on } {\partial\dom_\eps\cap\partial\dom}.
\end{align}
From the boundary condition $\tu_\ext=\eps g$ on $\partial\dom_\eps\cap\partial L_\eps$,
the recursive boundary conditions on $\Gamma$ for $\tu_{\ext,m,n}$ are derived in \eqref{eqn:tumn}.
The derivation of the boundary conditions  of $u_{\rin,m,n}$ from those of   $\tu_{\ext,m,n}$
is below.

For $m=0$, one has on $\Gamma$
\[
\begin{split}
u_{\rin,0,n}=\tu_{\ext,1,n}
=&\delta_{0,n} g-h\partial_{\bn}\tu_{\ext,0,n}\\
=&\begin{cases}
g-ch\partial_{\bn}u_{\rin,0,0}, \quad & n=0,\\
-ch\partial_{\bn}u_{\rin,0,n}-h\partial_{\bn}u_{\rin,0,n-1}, & n\geq 1.
\end{cases}
\end{split}
\]
Thus we obtain the following {\it Robin boundary conditions}
\[
u_{\rin,0,n}+ch\partial_{\bn}u_{\rin,0,n}=\begin{cases}
g, \quad & n=0,\\
-h\partial_{\bn}u_{\rin,0,n-1}, & n\geq 1.
\end{cases}
\]

For $m\geq 1$ and $n=0$, on $\Gamma$,
\[
\begin{split}
u_{\rin,m,0}=&\tu_{\ext,m+1,0}\\
=&\dfrac{h^{m}}{m!}\partial_{\bn}^{m} g-h\partial_{\bn}\tu_{\ext,m,0}
-\sum_{k=2}^{m+1}\dfrac{h^k}{k!}F_{k,\delta_{k,m+1}cf_\ext}
[\tu_{\ext,m+1-k,0},\partial_{\bn}\tu_{\ext,m+1-k,0}]\\
=&\dfrac{h^{m}}{m!}\partial_{\bn}^{m} g-ch\partial_{\bn}u_{\rin,m,0}
-\sum_{k=2}^{m}\dfrac{h^k}{k!}F_{k,0}
[u_{\rin,m-k,0},c\partial_{\bn}u_{\rin,m+1-k,0}]\\
&-\dfrac{h^{m+1}}{(m+1)!}F_{m+1,cf_{\ext}}
[0,c\partial_{\bn}u_{\rin,0,0}],
\end{split}
\]
hence the Robin boundary condition on $\Gamma$ is
\[
\begin{split}
u_{\rin,m,0}+ch\partial_{\bn}u_{\rin,m,0}=&\dfrac{h^{m}}{m!}\partial_{\bn}^{m} g
-\sum_{k=2}^{m}\dfrac{h^k}{k!}F_{k,0}
[u_{\rin,m-k,0},c\partial_{\bn}u_{\rin,m+1-k,0}]
\\&-\dfrac{h^{m+1}}{(m+1)!}F_{m+1,cf_{\ext}}
[0,c\partial_{\bn}u_{\rin,0,0}].
\end{split}
\]

For $m, n\geq 1$, on $\Gamma$, we have
\[
\begin{split}
u_{\rin,m,n}=&\tu_{\ext,m+1,n}\\
=&-h\partial_{\bn}\tu_{\ext,m,n}-\sum_{k=2}^{m+1}\dfrac{h^k}{k!}F_{k,\delta_{k,m+1}\delta_{1,n}f_\ext}
[\tu_{\ext,m+1-k,n},\partial_{\bn}\tu_{\ext,m+1-k,n}],\\
=&-ch\partial_{\bn}u_{\rin,m,n}-h\partial_{\bn}u_{\rin,m,n-1}\\
&-\sum_{k=2}^{m}\dfrac{h^k}{k!}F_{k,0}
[u_{\rin,m-k,n},c\partial_{\bn}u_{\rin,m+1-k,n}+\partial_{\bn}u_{\rin,m+1-k,n-1}]\\
&-\dfrac{h^{m+1}}{(m+1)!}F_{m+1,\delta_{1,n}f_{\ext}}
[0,c\partial_{\bn}u_{\rin,0,n}+\partial_{\bn}u_{\rin,0,n-1}],
\end{split}
\]
and thus the Robin boundary condition on $\Gamma$ is
\[
\begin{split}
&u_{\rin,m,n}+ch\partial_{\bn}u_{\rin,m,n}\\
=&-h\partial_{\bn}u_{\rin,m,n-1}
-\dfrac{h^{m+1}}{(m+1)!}F_{m+1,\delta_{1,n}f_\ext}[0,\partial_{\bn}u_{\rin,0,n-1}+c\partial_{\bn}u_{\rin,0,n}]\\
&-\displaystyle\sum_{k=2}^m\dfrac{h^k}{k!}F_{k,0}
[u_{\rin,m-k,n},\partial_{\bn}u_{\rin,m+1-k,n-1}+c\partial_{\bn}u_{\rin,m+1-k,n}].
\end{split}
\]

\medskip
To summarize the above three cases,
we find that the limit $ {\eps}/{\sigma}\to c$ is quite important:
the value of $c$ determines  the type of the boundary conditions
in the asymptotic series.
$c=0$ means  $\eps$ decays faster than $\sigma$
or $\sigma$ is not  a small value,
 and our result shows that the  boundary conditions
 for the asymptotic expansions  remain the  Dirichlet type.
$c=\infty$ corresponds to a very small conductivity
in the exterior layer,  and in this case, it is interesting to see the Neumann conditions on  $\partial\dom$
  for   all  terms in the asymptotic expansions.
    The case of  $c \in (0, \infty)$  that leads to the Robin boundary conditions
can be regarded as between the above two extreme cases.

\section{Asymptotic Expansion for the perturbed Interface Problem}\label{sec:5}
The previous sections on the interface problem
assume that the interface is the boundary of the fixed domain $\dom$.
The geometric perturbation is only applied to the outside layer.
In this section, we focus on the situation where
the interface is perturbed.
The setting is the following.
 Assume  $\dom$ is a smooth bounded domain and  is partitioned into two   subdomains separated by  an interface $\Gamma_\eps$:
\begin{equation}
\dom = \domep \cup  \domen \cup \Gamma_\eps.
\end{equation}
 $\Gamma_\eps=\partial \domen\cap\partial\domep$ is assumed smooth.
The interface $\Gamma_\eps$ is modelled in a perturbative way.
Assume there is  a fixed interface $\Gamma$ and let $\bn(\bx)$
be the unit normal vector on $\Gamma$ pointing  outward of  $\dom^-$.
That is, the whole domain $\dom$ has a fixed decomposition
$\dom=\dom^-\cup \dom^+\cup \Gamma$.
Then we define $\Gamma_\eps$ for $\eps < \eps_0$
\begin{equation}\label{eqn:Gammaeps}
\Gamma_\eps =\set{\bx': \bx' = \bx + \eps h(\bx) \bn(\bx), \bx \in \Gamma}.
\end{equation}

 We  consider  the following   interface problem on $\dom$ with
 transmission condition on the interface $\Gamma_\eps$:
\begin{equation} \label{EIP}
\begin{cases}
& -   \nabla\cdot \left (\sigma^\pm(\bx)\nabla u^\pm_\eps(\bx) \right) =  f(\bx) \quad \text{in } {\dom_\eps^\pm},\\
& {u^+_\eps(\bx)}={u^-_\eps(\bx)},  \quad \displaystyle {\sigma^+(\bx) \partial_{\bn_\eps} u^+_\eps(\bx)}=
 {\sigma^-(\bx) \partial_{\bn_\eps} u^-_\eps(\bx)}\quad \text{on } {\Ge}, \\
& u_\eps^\pm =g \quad \text{on } \partial\dom\cap\partial\dom_\eps^\pm,
\end{cases}
\end{equation}
where $\sigma^\pm(\bx)>0$ for every $\bx\in\dom_\eps^\pm$, and $\bn_\eps(\bx)$
is the unit normal vector on $\Gamma_\eps$ pointing  outward of  $\dom^-_\eps$.
Denote   $u_\eps$ restricted  on  $\domep$  and    $\domen$ by $u^+_\eps$ and $u^-_\eps$, respectively.
For this  interface problem  \eqref{EIP},
the variational formulation   reads as follows:
Seek $u_\eps\in H^1_0(\dom)$ such that
\begin{equation} \label{EIP-weak}
\int_{\domep}\! \sigma^+ \nabla u^+_\eps \cdot \nabla v \,\rd \bx +
\int_{\domen}\! \sigma^- \nabla u^-_\eps \cdot \nabla v \,\rd\bx =
\int_{\dom} \! f v\,\rd\bx ~~~, \forall  v \in H_0^1(\dom).
\end{equation}
We assume that $\sigma^\pm$ are defined on sufficiently large domains such that for every sufficiently small $\eps$,
$\sigma^\pm \in C^\infty (\dom_\eps^\pm)$.
We also assume
$f \in C^\infty (\dom)$.
Then,
$u^\pm_\eps \in C^{\infty}(\overline{\dom^\pm_\eps})$.

Assume the coefficient  $\sigma(\bx)$ is the piecewise homogeneous case:
\begin{equation} \label{def:a}
\sigma(\bx)=
\begin{cases}
1, &  \bx \in \domep, \\
\sigma, & \bx \in \domen,
\end{cases}
\end{equation}
where  $\sigma$ is a positive constant.   We are interested in 
the high-contrast ratio limit, which corresponds to
 a very small or very large value of $\sigma$.

\cite{HarbLi2013} has studied  the first order
and second order perturbations to the problem \eqref{EIP} by the
method of  shape calculus  for small $\eps$. The second order approximation
was obtained by considering the Hessian with respect to the perturbation function
$h$ on the reference interface $\Gamma$.
We  shall show how to derive the expansions   for small $\eps$ up to any order
by the method of Taylor expansion.
The main tool used here is similar to our previous work in  \cite{TPT_SA}
to calculate the first order derivative.
After deriving the $\eps$-expansion,
we proceed to the two-parameter expansion.

%

\subsection{Asymptotic expansions in $\eps$}
 \subsubsection{The extension of $u^\pm_\eps$}
 \begin{figure}
\includegraphics[width=0.42\textwidth]{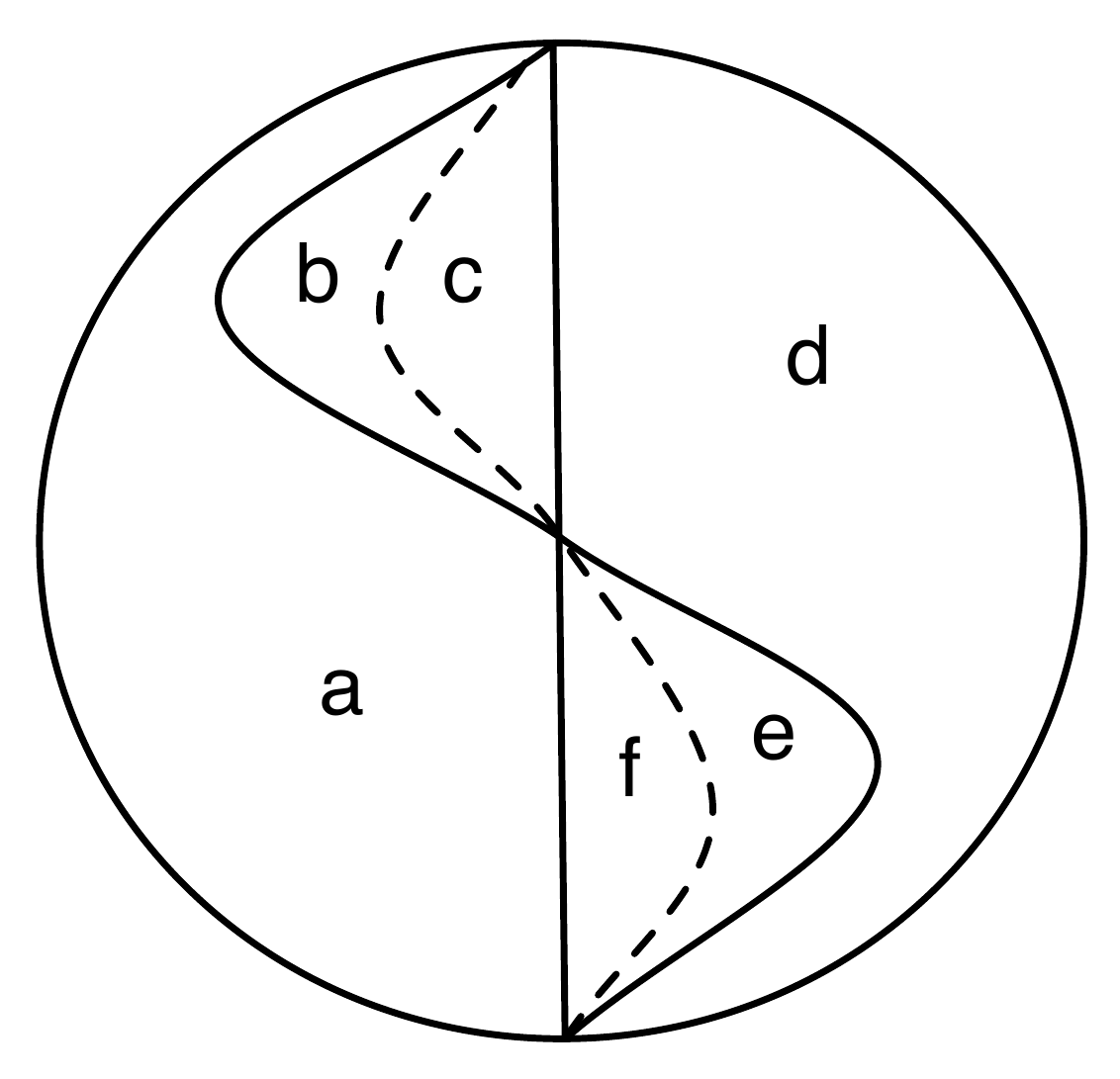}
\hfill
\includegraphics[width=0.41\textwidth]{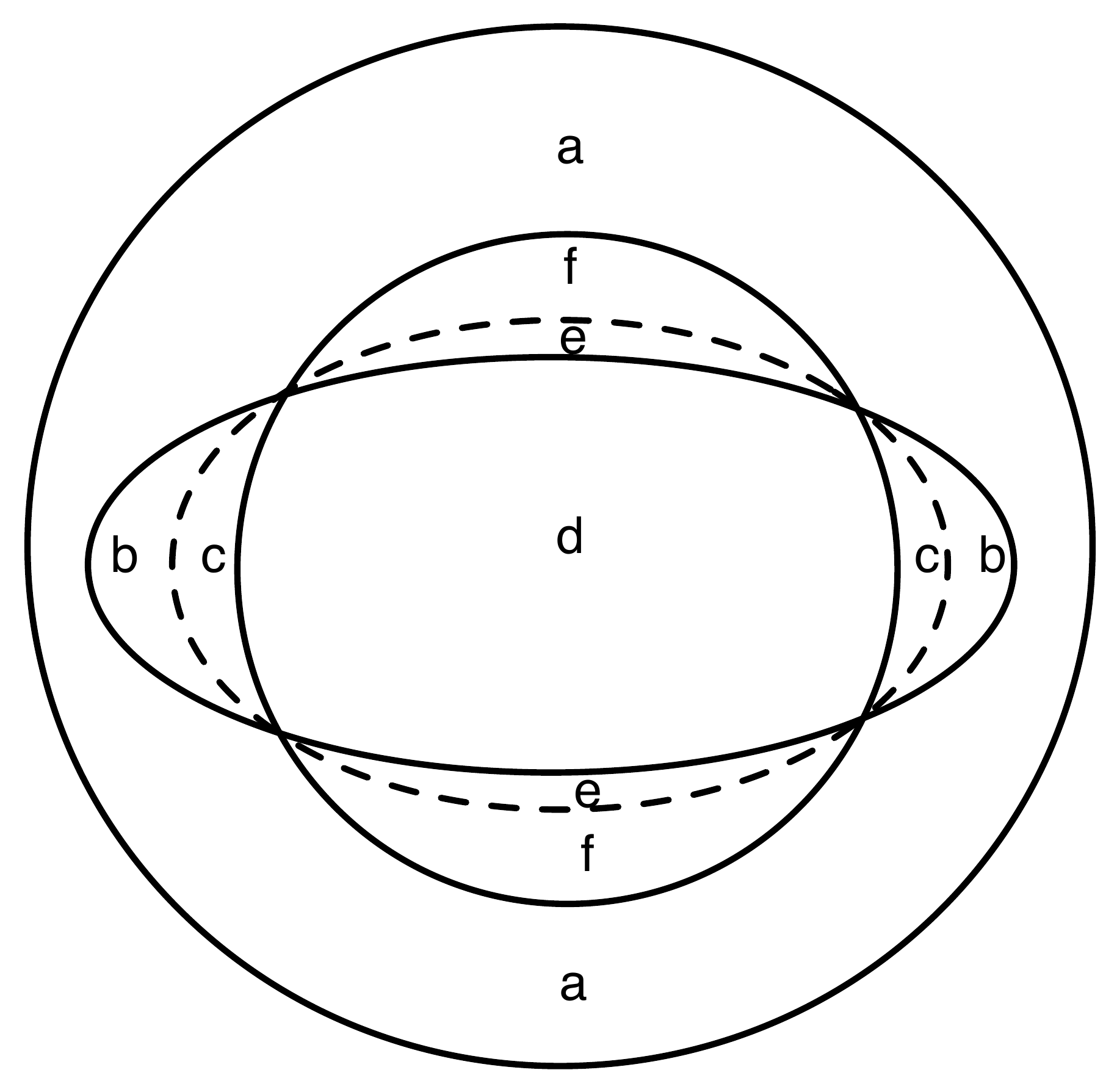}
\caption{Schematic illustration of the interface perturbations and extensions for two different cases
of the interface problem.
The unperturbed interfaces $\Gamma$ are
the vertical diameter  (left) and the inner circle (right) respectively, while
the perturbed interfaces $\Gamma_\eps$ are the dashed lines for both cases.
The unperturbed, $\eps$-perturbed and $\eps_0$-perturbed subdomains are respectively $\dom^+=a\cup b\cup c$, $\dom^-=d\cup e\cup  f$;
$\dom^+_\eps=a\cup b\cup f$, $\dom^-_\eps=d\cup c\cup e$; $\dom^+_{\eps_0}=a\cup e\cup f$, $\dom^-_{\eps_0}=d\cup b\cup c$.
$u^\pm_\eps$ are extended to the sufficiently large fixed domains $\dom^+\cup\dom^+_{\eps_0}=a\cup b\cup c\cup e\cup f$ and
$\dom^-\cup\dom^-_{\eps_0}=d\cup b\cup c\cup e\cup f$ respectively. The Cauchy problems for $\tu_\eps^\pm$ are
imposed in the thin layers $\triangle_\eps^+=c\cup e$ and $\triangle_\eps^-=b\cup f$ respectively.
}
\label{fig:interface perturbation}
\end{figure}

The first technical issue 
when applying the Taylor expansion is how to extend the  solutions 
$u^\pm_\eps$  
 of \eqref{EIP-weak} from their own subdomains $\dom_\eps^\pm$
onto    the  larger and fixed domains which both include
the interface $\Gamma_\eps$ for all $\eps\in [0,\eps_0]$. 
 Such   domains are chosen as  $\dom^\pm\cup\dom_{\eps_0}^\pm$.
 On these fixed domains $\dom^\pm\cup\dom_{\eps_0}^\pm$, $u^\pm_\eps$  are known on the parts  $\overline{\dom_\eps^\pm}$;
we thus consider  the differences   $ \triangle_{\eps}^\pm$ which
consist of the disjoint  thin layers:
\[
 \triangle_{\eps}^\pm:=(\dom^\pm\cup\dom_{\eps_0}^\pm)\setminus\overline{\dom_\eps^\pm} =(\dom^\pm\setminus\overline{\dom_\eps^\pm})\cup
 (\dom_{\eps_0}^\pm\setminus\overline{\dom_\eps^\pm}).
\]
Refer to  Figure  \ref{fig:interface perturbation}.
Denote the solution extended on $\triangle_{\eps}^\pm$ by  $\tu_\eps^\pm$,
and assume that  $\tu_\eps^\pm$  and $u_\eps^\pm$  have the same values and
the same normal derivatives on the common boundary $\Gamma_\eps$.
Specifically,  $\tu_\eps^\pm$ are constructed as  the unique solutions to  the following { Cauchy} problems     posed in the thin layers $\dom^\pm\setminus\overline{\dom_\eps^\pm}$ and
 $\dom_{\eps_0}^\pm\setminus\overline{\dom_\eps^\pm}$ respectively:
\begin{equation*}
\begin{cases}
 -   \nabla\cdot \left (\sigma(\bx)\nabla \tu_\eps^\pm(\bx) \right) =  f(\bx) \qquad   &\text{in } \triangle_{\eps}^\pm=(\dom^\pm\setminus\overline{\dom_\eps^\pm})\cup
 (\dom_{\eps_0}^\pm\setminus\overline{\dom_\eps^\pm}),\\
\tu_\eps^\pm=u_\eps^\pm, \quad \partial_{\bn}  \tu_\eps^\pm=\partial_{\bn} u_\eps^\pm\quad &\text{on } \Gamma_\eps,
\end{cases}
\end{equation*}
where $u_\eps^\pm$, the solution to equation \eqref{EIP-weak}, are presumably given.
As in Section \ref{subsubsec:extension}, the Cauchy-Kovalevskaya theorem  \cite{Evans:1998} guarantees that
such extensions can be realized analytically for sufficiently small $\eps_0$ so that the Taylor expansion can be applied in a neighbourhood of $\Gamma$.

\subsubsection{Asymptotic expansions on the fixed subdomains $\dom^\pm$}

For ease of notation, we will still use $u^\pm_\eps$ to denote their extensions
defined above.
Let us consider
\begin{equation}\label{eqn:ansatz}
u^\pm_\eps=u_0^\pm + \eps u_1^\pm  + \eps^2 u_2^\pm +   \ldots.
\end{equation}

First, for the transmission condition  $u_\eps^+=u_\eps^-$ on $\Gamma_\eps$ in \eqref{EIP},
by   \eqref{eqn:Gammaeps}  we have
\[
u^+_\eps(\bx+\eps h(\bx)\bn(\bx))=u^-_\eps(\bx+\eps h(\bx)\bn(\bx)), \quad \bx\in\Gamma,
\]
and by \eqref{eqn:ansatz}, we have
\[
\sum_{n=0}^\infty \eps^n u^+_n(\bx+\eps h(\bx)\bn(\bx))=\sum_{n=0}^\infty \eps^n u^-_n(\bx+\eps h(\bx)\bn(\bx)), \quad \bx\in\Gamma.
\]
Then the Taylor expansions in $\eps$ on both sides as before can yield
\[
\sum_{k=0}^n\frac{(h(\bx))^k}{k!}{\partial_{\bn}^k}u^+_{n-k}(\bx)=
\sum_{k=0}^n\frac{(h(\bx))^k}{k!}{\partial_{\bn}^k}u^-_{n-k}(\bx),\quad \bx\in\Gamma,
\]
thus
\begin{equation}\label{eqn:unjumpcond}
\sqbk{u_n(\bx)}=-\sum_{k=1}^n\frac{(h(\bx))^k}{k!}\sqbk{{\partial_{\bn}^k}u_{n-k}(\bx)} \quad \text{on } {\Gamma},
\end{equation}
where $\sqbk{v}:=v^+ - v ^-$ deontes the jump  across the subdomains from $\dom^-$ to $\dom^+$.

On the other hand, from the variational form \eqref{EIP-weak}, we obtain
\begin{equation}\label{eqn:chdom}
\begin{split}
&\int_{\dom^+} \sigma^+ \nabla u^+_\eps \cdot \nabla v \,\rd\bx +
\int_{\dom^-} \sigma^- \nabla u^-_\eps \cdot \nabla v \,\rd\bx\\
&- \int_{\delta\dom_\eps} \sigma^+ \nabla u^+_\eps \cdot \nabla v\,\rd\bx+ \int_{\delta\dom_\eps} \sigma^- \nabla u^-_\eps \cdot \nabla v\,\rd\bx=
\int_{\dom} f v\,\rd\bx,
\end{split}
\end{equation}
where $\delta\dom_\eps=(\dom_\eps^+\setminus\dom^+)\cup(\dom^+\setminus\dom_\eps^+)=(\dom^-\setminus\dom_\eps^-)\cup(\dom^-_\eps\setminus\dom^-)$,
and the integrand on $\delta \dom_\eps$
is taken with a minus sign over $\dom_\eps^+\setminus\dom^+=\dom^-\setminus\dom_\eps^-$
 and  a plus sign over $\dom^+\setminus\dom_\eps^+=\dom^-_\eps\setminus\dom^-$.
Note that
\[
\dom^+\setminus\dom_\eps^+=\dom^-_\eps\setminus\dom^-=\set{\bx+t\bn(\bx):\bx\in\Gamma,~h(\bx)>0, ~0\leq t<\eps h(\bx)},
\]
\[
\dom_\eps^+\setminus\dom^+=\dom^-\setminus\dom_\eps^-=\set{\bx+t\bn(\bx):\bx\in\Gamma,~h(\bx)<0, ~\eps h(\bx)< t\leq 0}.
\]

To handle the integration $\int_{\delta \dom_\eps}$ in \eqref{eqn:chdom}, we introduce the curvilinear coordinates $(\bxi,t)$  in a sufficiently small tubular neighborhood of $\Gamma$,
which are   defined by $$\bx =\btheta(\bxi)+t\bn(\bxi),$$ where $\bxi=(\xi_1,\cdots,\xi_{d-1})\in\Omega \longmapsto \btheta(\bxi)\in\Gamma$
is a parametrization of the interface $\Gamma$ and $\bn(\bxi):=\bn(\btheta(\bxi))$.
Then for any smooth function $g$, by making use of a change of variables, we have
\begin{equation}\label{eqn:chofvar}%
\begin{split}
\int_{\delta\dom_\eps} g(\bx)\,\rd\bx=&\int_{\Omega}\biggl(\int_0^{\eps h(\btheta(\bxi))}g(\btheta(\bxi)+t\bn(\bxi))\abs{J(\bxi,t)} \,\rd t\biggr)\rd {\bxi},
\end{split}
\end{equation}
where $J(\bxi,t)$ is the Jacobian determinant of the mapping $(\bxi,t)\longmapsto \bx$.
By Appendix \ref{sec:proof3},   \eqref{eqn:chofvar} is equivalent to 
\begin{equation}\label{eqn:intdeldomgen}
\int_{\delta\dom_\eps} g(\bx)\,\rd\bx=\int_{\Gamma}\biggl(\int_0^{\eps h(\btheta)}\tilde{g}(\btheta,t)\det(I+tW) \,\rd t\biggr)\rd S_\Gamma({\btheta}),
\end{equation}
where $\tilde{g}(\btheta,t):=g(\btheta+t\bn(\btheta))$, $\btheta\in\Gamma$,
 $I$ denotes the identity matrix, $W=(W_i^{~j})$ is the matrix representation
of the   Weingarten map,
 and $\rd S_{\Gamma}(\btheta)$ denotes  the surface area element on the hypersurface $\Gamma$.

The equality  \eqref{eqn:intdeldomgen} is the major foundation to apply 
the asymptotic expansion. We show how to proceed this task
by considering the first two orders $u^\pm_0$ and $u^\pm_1$. 
Since \begin{equation}\label{232}
\tilde{g}(\btheta,t)\det(I+tW)=\tilde{g}(\btheta,0)+\Od(t)=g(\btheta)+\Od(t),
\end{equation}
 then
\begin{equation}\label{eqn:intdeldom01}
\begin{split}
\int_{\delta\dom_\eps} g(\bx)\,\rd\bx
=&~\eps\int_\Gamma
h(\btheta)g(\btheta)\,\rd S_\Gamma({\btheta})+\Od(\eps^2).
\end{split}
\end{equation}
Note that  on $\Gamma$, we have the following orthogonal decomposition of the gradient operator:
\[
\nabla=\nabla_\Gamma+\bn\partial_{\bn},
\]
where $\nabla_\Gamma$ denotes the surface gradient operator.
Then
\begin{equation}\label{eqn:nabladotnabla}
\sqbk{\sigma\nabla  u_k}\cdot\nabla v=\sqbk{\sigma\nabla_\Gamma u_k} \cdot\nabla_\Gamma v+\sqbk{\sigma\partial_{\bn} u_k}\partial_{\bn}v\quad \text{on }\Gamma.
\end{equation}
Substituting \eqref{eqn:ansatz} into \eqref{eqn:chdom}, and applying \eqref{eqn:intdeldom01} and \eqref{eqn:nabladotnabla}, we are led to
\[
\begin{split}
&\sum_{n=0}^1\eps^n\int_{\dom^+} \sigma^+ \nabla u^+_n \cdot \nabla v \,\rd\bx +\sum_{n=0}^1\eps^n
\int_{\dom^-} \sigma^- \nabla u^-_n \cdot \nabla v \,\rd\bx\\
&- \eps\int_{\Gamma}h \bigl(\sqbk{\sigma\nabla_\Gamma u_0} \cdot\nabla_\Gamma v+\sqbk{\sigma\partial_{\bn} u_0}\partial_{\bn}v\bigr)\,\rd S_\Gamma
+\Od(\eps^2)=
\int_{\dom} f v\,\rd\bx.
\end{split}
\]
Now we collect terms with equal powers of $\eps$ and obtain:
\begin{equation}\label{eqn:u0weak}
\int_{\dom^+} \sigma^+ \nabla u^+_0 \cdot \nabla v \,\rd\bx +
\int_{\dom^-} \sigma^- \nabla u^-_0 \cdot \nabla v \,\rd\bx=
\int_{\dom} f v\,\rd\bx;
\end{equation}
\begin{equation}\label{eqn:u1weak}
\begin{split}
&\int_{\dom^+} \sigma^+ \nabla u^+_1 \cdot \nabla v \,\rd\bx +
\int_{\dom^-} \sigma^- \nabla u^-_1 \cdot \nabla v \,\rd\bx\\
&- \int_{\Gamma}h\bigl(\sqbk{\sigma \nabla_\Gamma u_{0}} \cdot \nabla_\Gamma v
+\sqbk{\sigma \partial_{\bn} u_{0}}\partial_{\bn} v\bigr)\,\rd S_{\Gamma}=0 \quad \forall   v \in H_0^1(\dom).
\end{split}
\end{equation}
These weak formulations  together with \eqref{eqn:unjumpcond} with $n=0, 1$, lead to  the following two PDEs  for $u_0$ and $u_1$, respectively:
\begin{equation}\label{eqn:u0PDEpbm}
\begin{cases}
& -   \nabla\cdot \left (\sigma(\bx)\nabla u_0(\bx) \right) =  f(\bx)  \quad \text{in } {\dom^+} \cup \dom^-,\\
& \sqbk{u_0(\bx)}=   \displaystyle \sqbk{ \sigma(\bx) \partial_{\bn} u_0(\bx) }=0  \quad \text{on } {\Gamma}, \\
& u =g \quad \text{on } \partial\dom.
\end{cases}
\end{equation}
and
\begin{equation}\label{eqn:u1PDE}
\begin{cases}
& -   \nabla\cdot \left (\sigma(\bx)\nabla u_1(\bx) \right) = 0  \quad \text{in } {\dom^+} \cup \dom^-,\\
& \sqbk{u_1}=-h\sqbk{\partial_{\bn}u_0},\quad   \displaystyle \sqbk{ \sigma\partial_{\bn} u_1 }=
\nabla_\Gamma\cdot (h\sqbk{\sigma} \nabla_\Gamma u_{0}) \quad \text{on } {\Gamma}, \\
& u_1 =0 \quad \text{on } \partial\dom.
\end{cases}
\end{equation}

The equations for higher order terms, $u_{n}, n\geq 2$,
can be  derived in the same way by considering the higher order Taylor approximations for  \eqref{232}.

\subsection{Two-parameter expansions}

The expansion of high-contrast ratio without interface perturbation is derived  in \cite{CEG2014}.
In the sequel, we show the two-parameter expansion results by combining our $\eps$ expansion and
the $\sigma$-expansion in  \cite{CEG2014}. We need to consider the following
two different cases:
\begin{itemize}
\item[(i)] $\eps\rightarrow 0,\sigma\rightarrow \infty$,
\item[(ii)] $\eps\rightarrow 0,\sigma\rightarrow 0$.
\end{itemize}
Note that $\sigma$ is defined on the subdomain $\dom_\eps^-$ by \eqref{def:a}.
For Case (i), the solution is still bounded;
 for Case (ii), the solution on $\dom^-_\eps$
behaves at the order $\Od(1/\sigma)$.
The difference between these two cases
is mainly a scaling factor $1/\sigma$.
We  focus on   Case (i) here.
The derivation for Case (ii) can be found in Appendix \ref{sec:case}.

In Case (i), we have $\eps\to0$ and $\sigma\to+\infty$.
We introduce $\mu=1/\sigma$  and treat $\eps$ and $\mu$ as independent small parameters.
Assume $u^+_\eps$ and $u^-_\eps$  have double asymptotic expansions
\[
u^\pm_\eps(\bx)=\sum_{m,n=0}^\infty u^\pm_{m,n}(\bx)\eps^m\mu^n,  \quad \bx\in\dom^\pm,
\]
and introduce the notation
\begin{equation}\label{eqn:subansatz}
u^\pm_{m,\cdot}(\bx)=\sum_{n=0}^\infty u^\pm_{m,n}(\bx)\mu^n, \quad \bx\in\dom^\pm,
\end{equation}

From the condition \eqref{eqn:unjumpcond} we obtain
\begin{equation}
\sqbk{u_{0,\cdot}}=0,\quad \sqbk{u_{1,\cdot}}=-h\sqbk{{\partial_{\bn}}u_{0,\cdot}} \quad \text{on } {\Gamma},\label{eqn:u0+1bdc}
\end{equation}

Substituting  
\eqref{eqn:subansatz} into \eqref{eqn:u0+1bdc} 
and
matching the terms with the same order of $\mu$ yield that
\begin{align}
&u^-_{0,n}=u^+_{0,n},   \quad n\geq 0,\label{eqn:u^-_{0,n}bdc}\\
&u^-_{1,n}=u^+_{1,n}+h\sqbk{{\partial_{\bn}}u_{0,n}}  \quad n\geq 0.
\label{eqn:u^-_{1,n}bdc}
\end{align}
For the piecewisely homogeneous case of  $\sigma$ considered here, 
 \eqref{eqn:u0weak} and \eqref{eqn:u1weak} become  that   for all $v\in H_0^1(\dom)$,
\begin{align}
&\int_{\dom^+} \nabla u^+_{0,\cdot} \cdot \nabla v \,\rd\bx +
\int_{\dom^-} \sigma \nabla u^-_{0,\cdot} \cdot \nabla v \,\rd\bx=
\int_{\dom} f v\,\rd\bx;\label{eqn:u_{0,}weak}\\
&\int_{\dom^+} \nabla u^+_{1,\cdot} \cdot \nabla v \,\rd\bx +
\int_{\dom^-} \sigma\nabla u^-_{1,\cdot} \cdot \nabla v \,\rd\bx\nonumber\\
&- \int_{\Gamma}h\bigl[(\nabla_{\Gamma} u^+_{0,\cdot}-\sigma \nabla_{\Gamma} u^-_{0,\cdot}) \cdot \nabla_\Gamma v
+(\partial_{\bn}u^+_{0,\cdot}-\sigma \partial_{\bn} u^-_{0,\cdot})\partial_{\bn} v\bigr]\,\rd S_{\Gamma}=0. \label{eqn:u_{1,}weak}
\end{align}
Substituting  
\eqref{eqn:subansatz} into 
\eqref{eqn:u_{0,}weak} and \eqref{eqn:u_{1,}weak},
we have that   for all $v\in H_0^1(\dom)$,
\begin{equation}\label{eqn:u_{0,0}weak}
\int_{\dom^-} \nabla u^-_{0,0} \cdot \nabla v \,\rd\bx=0,
\end{equation}
\begin{equation}\label{eqn:u_{1,0}weak}
\int_{\dom^-} \nabla u^-_{1,0} \cdot \nabla v \,\rd\bx=
\int_{\Gamma}h\bigl[\nabla_{\Gamma} u^-_{0,0} \cdot \nabla_\Gamma v
+\partial_{\bn} u^-_{0,0}\partial_{\bn} v\bigr]\,\rd S_{\Gamma},
\end{equation}
and for $n\geq 0$,
\begin{align}
&\int_{\dom^+} \nabla u^+_{0,n} \cdot \nabla v \,\rd\bx +
\int_{\dom^-} \nabla u^-_{0,n+1} \cdot \nabla v \,\rd\bx=
\delta_{0,n}\int_{\dom} f v\,\rd\bx;\label{eqn:u_{0,n}weak1}\\
&\int_{\dom^+} \nabla u^+_{1,n} \cdot \nabla v \,\rd\bx +
\int_{\dom^-} \nabla u^-_{1,n+1} \cdot \nabla v \,\rd\bx\nonumber\\
&- \int_{\Gamma}h\bigl[(\nabla_{\Gamma} u^+_{0,n}-\nabla_{\Gamma} u^-_{0,n+1}) \cdot \nabla_\Gamma v
+(\partial_{\bn}u^+_{0,n}- \partial_{\bn} u^-_{0,n+1})\partial_{\bn} v\bigr]\,\rd S_{\Gamma}=0. \label{eqn:u_{1,n}weak1}
\end{align}

From   \eqref{eqn:u^-_{0,n}bdc} and the weak formulation
\eqref{eqn:u_{0,0}weak} \eqref{eqn:u_{0,n}weak1}, we have
the following PDEs for each term:
\begin{equation}\label{eqn:u^-_00PDE}
\begin{cases}
& -\Delta u^-_{0,0}=0 \quad \text{in } \dom^-,\\
& \partial_{\bn}u^-_{0,0}=0 \quad \text{on }\Gamma,\\
& u^-_{0,0}=g\quad \text{on }\partial\dom^-\cap\partial\dom;
\end{cases}
\end{equation}
and for $n\geq 0$,
\begin{equation}\label{eqn:u^+_0nPDE}
\begin{cases}
& -\Delta u^+_{0,n}=\delta_{0,n}f \quad \text{in } \dom^+,\\
&u^+_{0,n}= u^-_{0,n} \quad \text{on }\Gamma,\\
& u^+_{0,n}=\delta_{0,n}g\quad \text{on }\partial\dom^+\cap\partial\dom;
\end{cases}
\end{equation}
and for $n\geq 1$,
\begin{equation}\label{eqn:u^-_0nPDE}
\begin{cases}
& -\Delta u^-_{0,n}=\delta_{1,n}f \quad \text{in } \dom^-,\\
& \partial_{\bn}u^-_{0,n}=\partial_{\bn}u^+_{0,n-1} \quad \text{on }\Gamma,\\
& u^-_{0,n}=0\quad \text{on }\partial\dom^-\cap\partial\dom.
\end{cases}
\end{equation}

We  also list the PDEs for the terms with $m=1$:
\begin{equation*}
\label{eqn:u^-_10PDE}
\begin{cases}
& -\Delta u^-_{1,0}=0 \quad \text{in } \dom^-,\\
& \partial_{\bn}u^-_{1,0}=-\nabla_\Gamma\cdot(h\nabla_\Gamma u^-_{0,0})\quad \text{on }\Gamma,\\
& u^-_{1,0}=0\quad \text{on }\partial\dom^-\cap\partial\dom;
\end{cases}
\end{equation*}
and for $n\geq 0$,
\begin{equation*}
\begin{cases}
& -\Delta u^+_{1,n}=0 \quad \text{in } \dom^+,\\
&u^+_{1,n}= u^-_{1,n}-h(\partial_{\bn}u^+_{0,n}-\partial_{\bn}u^-_{0,n})\quad \text{on }\Gamma,\\
& u^+_{1,n}=0\quad \text{on }\partial\dom^+\cap\partial\dom;
\end{cases}
\end{equation*}
and
\begin{equation*}
\begin{cases}
& -\Delta u^-_{1,n}=0 \quad \text{in } \dom^-,\\
& \partial_{\bn}u^-_{1,n}=\partial_{\bn}u^+_{1,n-1}
+\nabla_{\Gamma}\cdot\bigl[h(\nabla_{\Gamma} u^+_{0,n-1}-\nabla_{\Gamma} u^-_{0,n})\bigr] \quad \text{on }\Gamma,\\
& u^-_{1,n}=0\quad \text{on }\partial\dom^-\cap\partial\dom.
\end{cases}
\end{equation*}

\medskip
Here we need to pay attention to  a special  situation 
that  $\partial\dom^-=\Gamma$, or equivalently, $\partial\dom^-\cap\partial\dom=\emptyset$.
Refer to the right panel in Figure \ref{fig:interface perturbation}.
The boundary value problems above then may  become  Neumann problems, which are uniquely solvable only up to an arbitrary constant.
To determine those constants, as we have done in Section \ref{subsubsec:Case 2.2}, we  
need  the solvability condition from the next order. For $m=0$, and $n\geq 1$,  the solvability condition for  \eqref{eqn:u^-_0nPDE}   reads 
\begin{equation}
\label{425}
\int_{\Gamma}\partial_{\bn}u^+_{0,n-1}\,\rd S_{\Gamma}=-\delta_{1,n}\int_{\dom^-}f\,\rd \bx.
\end{equation}
 \eqref{eqn:u^-_00PDE} shows that $u^-_{0,0} \equiv C_0$.
 To determine $C_0$, we  need to look at $u^+_{0,0}$, which  satisfies   
\begin{equation}
\begin{cases}
& -\Delta u^+_{0,0}=f \quad \text{in } \dom^+,\\
&u^+_{0,0}= C_0 \quad \text{on }\Gamma,\\
&u^+_{0,0}= g \quad \text{on }\partial \dom,\\  
\end{cases}
\end{equation}
 by \eqref{eqn:u^+_0nPDE}.
By the solvability condition \eqref{425}, we have
$ \int_{\Gamma}\partial_{\bn}u^+_{0,0}\,\rd S_{\Gamma}=-\int_{\dom^-}f\,\rd \bx$,
which uniquely determines the constant 
\[ C_0 = -\frac{
 \int_{\dom^-} f +  \int_{\Gamma}(\partial_{\bn} \phi_2  +\partial_{\bn} \phi_3)
 }
{ \int_{\Gamma}\partial_{\bn} \phi_1}
,\]
where $\phi_i$   solve the following equations, respectively,
\begin{equation}
\begin{cases}
& -\Delta \phi_1=0 \quad \text{in } \dom^+,\\
& \phi_1= 1 \quad \text{on }\Gamma,\\
&\phi_1= 0 \quad \text{on }\partial \dom,\\  
\end{cases}
~~ 
\begin{cases}
& -\Delta \phi_2=f \quad \text{in } \dom^+,\\
& \phi_2= 0 \quad \text{on }\Gamma,\\
&\phi_2= 0 \quad \text{on }\partial \dom,\\  
\end{cases}
~~ 
\begin{cases}
& -\Delta \phi_3=0 \quad \text{in } \dom^+,\\
& \phi_3= 0 \quad \text{on }\Gamma,\\
&\phi_3= g \quad \text{on }\partial \dom.\\  
\end{cases}
\end{equation}

\bigskip
\noindent \textbf{Acknowledgment.} J. Chen acknowledges support from National Natural Science Foundation of China grant 21602149.
L. Lin and X. Zhou acknowledge the financial support of Hong Kong GRF (109113, 11304314, 11304715).
Z. Zhang acknowledges the financial support of Hong Kong RGC grants (27300616, 17300817) and
National Natural Science Foundation of China via grant 11601457.
J. Chen would like to thank the hospitality of Department of Mathematics, City University of Hong Kong where part of the work was done.

\providecommand{\bysame}{\leavevmode\hbox to3em{\hrulefill}\thinspace}
\providecommand{\MR}{\relax\ifhmode\unskip\space\fi MR }
\providecommand{\MRhref}[2]{%
  \href{http://www.ams.org/mathscinet-getitem?mr=#1}{#2}
}
\providecommand{\href}[2]{#2}

 \appendix

\bigskip
 \section{Examples and generalizations for Section \ref{sec:2}}
\label{sec:example}

\subsection{Two examples}
\label{ssub:Twoex}
The following two 2D examples demonstrate
how the explicit form of the operator $F_{k,\Gamma,\Lo,f}$ can be obtained
in Lemma \ref{lem:convert}.

\begin{example}\label{ex:rectangle domain}
Consider
\[
\Lo u = -\sigma\partial_{x}^2 u-\sigma\partial_{y}^2u+cu
\]
over $\dom = (0,~L_1)\times(0,~L_2)$ with constants $\sigma>0$, $c\geq 0$
and the homogeneous Dirichlet boundary condition $g=0$.
Suppose the domain perturbation is only applied to the right boundary $\Gamma = \set{(L_1,y): 0\leq y\leq L_2}$
in the following form
 \begin{equation} \label{def:Gma}
 \Gamma_\eps = \set{(L_1+\eps h(y),y):0\leq y\leq L_2}.
 \end{equation}
Then, $\partial_{\bn}^k=\partial_x^k$, $\forall k\geq 0$.
The conversion of all partial derivatives of a function $w(x,y)$ with respect
to $x$ with $k\geq 2$  to the partial derivatives with respect to $y$ relies on
the repeated use of the partial differential equation
\begin{equation}\label{eqn:wphi}
\Lo w=-\sigma\partial_{x}^2w-\sigma\partial_{y}^2w+cw
=\phi.
\end{equation}
Explicit forms when $k=2$ and $k=3$ are as follows.
\begin{enumerate}
\item[$k=2$:] From \eqref{eqn:wphi}, we have
\begin{equation}
\label{eqn:dx2}
\partial_{x}^2w=-\frac \phi{\sigma}+ \frac {cw}{\sigma}-\partial_{y}^2w,
\end{equation}
which implies
\[
F_{2,\phi}[g_0,g_1]=-\frac \phi{\sigma}+ \frac {cg_0}{\sigma}-\partial_{y}^2g_0.
\]
\item[$k=3$:]
Differentiating \eqref{eqn:wphi} with respect to $x$ yields
\[
-\sigma\partial_{x}^3w-\sigma\partial_{y}^2\partial_xw+c\partial_x w
=\partial_x \phi,
\]
and thus
\begin{equation}
\label{eqn:dx3}
\partial_{x}^3w=\frac{1}{\sigma}\left(-\sigma\partial_{y}^2\partial_xw+c\partial_x w
-\partial_x\phi\right),
\end{equation}
which implies
\[
F_{3,\phi}[g_0,g_1]=\frac{1}{\sigma}\left(-\sigma\partial_{y}^2g_1+cg_1
-\partial_x\phi\right).
\]
\end{enumerate}

Using the above formulas, we can convert all partial derivatives of $u_0$ and $u_1$ with $k=2$
and $k=3$ in \eqref{eqn:w20z} and \eqref{eqn:w30z} to the partial derivatives with respect to $y$
with the following explicit forms.
\begin{enumerate}
\item[$k=2$:] Using $u_0(L_1,y)=0, \;0\leq y\leq L_2$ for $u_0$ in \eqref{eqn:dx2},  we have
\[
\partial_{x}^2u_0(L_1,y) = -\frac{1}{\sigma} f(L_1,y).
\]
Solving $\Lo u_1=0$ produces
\[
\partial_x^2 u_1(L_1,y) = \frac{cu_1(L_1,y)}{\sigma}
- \partial^2_y u_1(L_1,y).
\]
\item[$k=3$:]Substituting $u_0(x,y)$ into \eqref{eqn:dx3} and evaluating at $x=L_1$ yield
\[
\partial_{x}^3u_0(L_1,y)=\frac{1}{\sigma}\left(-\sigma\partial_{y}^2\partial_xu_0(L_1,y)+c\partial_xu_0(L_1,y)
-\partial_x f(L_1,y)\right).
\]
\end{enumerate}
%
\end{example}

\begin{example}\label{ex:Laplace}
Consider
$$\Lo=-\nabla\cdot\bigl(\sigma(\bx)\nabla)\bigr)=-\sigma\Delta-\nabla\sigma\cdot\nabla$$
in 2D with the scalar-valued {smooth} function $\sigma>0$ and set $g=0$.
Assume the 1D boundary $\partial \dom$ has a parametrization
by the arc length $s\mapsto\btheta(s)$.  Then at each point $\btheta(s)$, the unit tangent vector $\btau(s)$ is $\btheta'(s)$ and the
curvature $\kappa(s)$ is defined as
\[ \btau'(s)=-\kappa(s)\bn(s), ~\text{ equivalently, }~\kappa(s)=\btau(s)\cdot \bn'(s).\]
In a sufficiently small tubular neighborhood of $\partial \dom$,
the curvilinear coordinates $(s,t)$ are uniquely defined by $\bx=\btheta(s)+t\bn(s)$.
The gradient and the Laplace operators in curvilinear coordinates are
\[
\nabla=\btau(s)\frac{1}{1+t\kappa(s)}\partial_s+\bn(s)\partial_t,
\]
and
\[
\Delta=\frac{1}{1+t\kappa(s)}\partial_s\left(\frac{1}{1+t\kappa(s)}\partial_s\right)+\frac{\kappa(s)}{1+t\kappa(s)}\partial_t+\partial_t^2,
\]
respectively.

Then the operator $\Lo$ has the new form in terms of $(s,t)$,
\begin{equation}\label{eqn:lapst}
\begin{split}
\Lo u= & -\frac{\sigma}{1+t\kappa(s)}\partial_s\left(\frac{1}{1+t\kappa(s)}\partial_s u\right)
-\frac{\sigma \kappa(s)}{1+t\kappa(s)}\partial_t
u
\\
& -\sigma \partial_t^2 u
-\frac{\partial_t \sigma \partial_s u }{1+t\kappa(s)}-\partial_{\bn}\sigma \partial_t u.
\end{split}
\end{equation}
Note that  $\partial_{\bn}^k=\partial_{t}^k$, $\forall k\geq 0$.

Explicit forms for $F_{2,\phi}[\cdot,\cdot]$ and $F_{3,\phi}[\cdot,\cdot]$ are as follows.
\begin{enumerate}
\item [$k=2$:] Consider  the equation $\Lo w=\phi$ on $\partial\dom$
(i.e.,  $t=0$ )
\[
-\sigma(\partial^2_s w+\kappa\partial_{\bn}w+\partial_{\bn}^2w)-\partial_s\sigma\partial_sw-\partial_{\bn}\sigma\partial_{\bn}w
=\phi,
\]
which implies
\begin{equation}\label{eqn:partialbn2w}
\partial_{\bn}^2w=-\frac\phi\sigma-\frac{\partial_s\sigma}{\sigma}\partial_{s}w
-\frac{\partial_{\bn}\sigma}{\sigma}\partial_{\bn}w-\partial^2_s w-\kappa\partial_{\bn}w.
\end{equation}
Thus
\[
F_{2,\phi}[g_0,g_1]=-\frac\phi\sigma-\frac{\partial_s\sigma}{\sigma}\partial_{s}g_0
-\frac{\partial_{\bn}\sigma}{\sigma}g_1-\partial^2_s g_0-\kappa g_1.
\]
\item [$k=3$:] For simplicity we assume  constant coefficient $\sigma$
in the following calculation. Differentiating the equation $\Lo w=\phi$ with respect to $t$ at $t=0$ yields
\[
\begin{split}
&\sigma(\kappa'\partial_s+2\kappa\partial_s^2
-\partial_s^2\partial_t
+\kappa^2\partial_t-\kappa\partial_t^2-\partial_t^3)w=
\partial_t\phi.
\end{split}
\]
Thus,
\begin{equation}\label{eqn:partialbn3w}
\begin{split}
\partial_{\bn}^3 w=&(\kappa'\partial_s+2\kappa\partial_s^2
-\partial_s^2\partial_{\bn}
+\kappa^2\partial_{\bn}-\kappa\partial_{\bn}^2)w-\frac{\partial_{\bn}\phi}{\sigma}\\
=&(\kappa'\partial_s+3\kappa\partial_s^2
-\partial_s^2\partial_{\bn}
+2\kappa^2\partial_{\bn})w-\frac{\partial_{\bn}\phi}{\sigma}
+\frac{\kappa\phi}\sigma.
\end{split}
\end{equation}
In the last equality, we use \eqref{eqn:partialbn2w} for $\partial_{\bn}^2w$.
Therefore
\[
F_{3,\phi}[g_0,g_1]=(\kappa'\partial_s+3\kappa\partial_s^2)g_0+(2\kappa^2-\partial_s^2)g_1
+\frac{\kappa-\partial_{\bn}}{\sigma}\phi.
\]
\end{enumerate}

Using the above formulas, we can convert all partial derivatives of $u_0$ and $u_1$ with $k=2$
and $k=3$ in \eqref{eqn:w20z} and \eqref{eqn:w30z} to the partial derivatives with respect to $y$
with the following explicit forms
\begin{enumerate}
\item[$k=2$:] On $\partial\dom$, from \eqref{eqn:partialbn2w}, we have
\begin{equation*}\label{eqn:dn2u0}
\partial_{\bn}^2u_0 = -\frac{f}\sigma-\frac{\partial_{\bn}\sigma}{\sigma}\partial_{\bn}u_0-\kappa\partial_{\bn}u_0.
\end{equation*}
Thus the boundary condition \eqref{eqn:w20z}  for $u_2$ on $\partial\dom$ is  reduced to
\begin{equation}\label{eqn:u2bdcr}
u_2=-h\partial_{\bn}u_1+\frac{h^2}{2}\Bigl(\frac{\partial_{\bn}\sigma}{\sigma}+\kappa\Bigr)\partial_{\bn}u_0+\dfrac{h^2f}{2\sigma}.
\end{equation}
Similarly, for $n\geq 1$, we have on $\partial\dom$
\begin{equation*}\label{eqn:dn2un}
\partial_{\bn}^2u_n = -\kappa\partial_{\bn}u_n-\frac{\partial_{\bn}\sigma}{\sigma}\partial_{\bn}u_n
-\frac{\partial_{s}\sigma}{\sigma}\partial_{s}u_n-\partial^2_s u_n.
\end{equation*}
\item [$k=3$:] From \eqref{eqn:partialbn3w}, we have on $\partial\dom$
\[
\begin{split}
\partial_{\bn}^3 u_n=&F_{3,\delta_{0,n}f}[u_n,\partial_{\bn}{u_n}] \\
=&\delta_{0,n}\frac{\kappa-\partial_{\bn}}{\sigma}f
+(1-\delta_{0,n})\Bigl(\kappa'\partial_s+3\kappa\partial_s^2
\Bigr)u_n
+\Bigl(2\kappa^2
-\partial_s^2\Bigr)\partial_{\bn}u_n.
\end{split}
\]
\end{enumerate}
\end{example}

\subsection{Neumann boundary conditions}

If the Neumann boundary condition $\partial_{\bn} u_\eps=g$
rather than the Dirichlet boundary condition is prescribed on the boundary $\partial\dom_\eps$ for the equation  \eqref{eqn:uePDE},
the above method in $\S$\ref{sssec:aymp}  still works straightforwardly.
$w_\eps=g$  in \eqref{eqn:wBC} becomes $\partial_{\bn}w_\eps=g$ on $\partial\dom_\eps$ now.
 So  \eqref{eqn:swnbdry} becomes
   \[
   \sum_{n=0}^\infty  \eps^n \partial_{\bn}w_{n}\bigl(\bx+\eps h(\bx)\bn(x)\bigr)=g\bigl(\bx+\eps h(\bx)\bn(x)\bigr) \quad \text{for } \bx
  \in\partial\dom.
  \]
   The Taylor expansions  \eqref{eqn:Taylor} and \eqref{eqn:gTaylor} are still applicable along the normal direction $\bn$
   and the new conditions corresponding to \eqref{eqn:wmbdc} can be obtained   by
   following the previous procedure there.

For ease of exposition, let us just show the specific forms for   Example \ref{ex:rectangle domain}
when  the homogeneous Neumann boundary condition is imposed on $\Gamma_\eps$ defined in \eqref{def:Gma}.
The unit normal vector $\bn(x,y)$ on $\Gamma_\eps$    parallels to $(1,-\eps h'(y))$,
thus $\sum_{n=0}^\infty  \eps^n \partial_{\bn}w_{n}(x,y)=0 $ on $\Gamma_\eps$ is written as
\[
\sum_{n=0}^\infty  \eps^n \Bigl(\partial_x w_{n}\bigl(L_1+\eps h(y),y\bigr)-\eps h'(y) \partial_y w_{n}\bigl(L_1+\eps h(y),y\bigr)\Bigr)=0.
\]
Then the Taylor expansion  gives arise to
\[
\begin{split}
 &\sum_{m=0}^\infty   \eps^{m}  \sum_{k=0}^m   \frac{\bigl(h(y)\bigr)^k}{k!}\partial_x^{k+1} w_{m-k}(L_1,y)
=
\\
&~\qquad\sum_{m=1}^\infty  \eps^{m} \sum_{k=0}^{m-1} \frac{\bigl(h(y)\bigr)^k}{k!} h'(y)  \partial_x^k \partial_y  w_{m-1-k}(L_1,y) .
\end{split}
\]
Then after matching each order $\eps^m$, we have that
\begin{equation}\label{eqn:w0nu}
\partial_x w_0(L_1,y) = 0,
\end{equation}
and for $ m\geq 1$,
\[
  \sum_{k=0}^m   \frac{\bigl(h(y)\bigr)^k}{k!}\partial_x^{k+1} w_{m-k}(L_1,y)
=   \sum_{k=0}^{m-1} \frac{\bigl(h(y)\bigr)^k}{k!} h'(y)  \partial_x^k \partial_y  w_{m-1-k}(L_1,y).
\]
In particular, the boundary conditions for $m=1,2$ are
\begin{align*}
& \partial_x w_1 +h(y)\partial_x^2 w_0 -
h'(y) \partial_y w_0=0,  \\
&
\partial_x w_2  +h(y)\partial_x^2 w_1 +
\frac12 \bigl(h(y)\bigr)^2\partial_x^{3} w_0
- h'(y)\partial_y w_1-h(y)h'(y)\partial_x\partial_y w_0=0.
\end{align*}

\subsection{Nonlinear equations}
For some nonlinear partial differential equations, we may still
use the above Taylor expansion method in Section \ref{sssec:aymp}
to derive a sequence of  $u_n$ in the asymptotic expansion.
We illustrate this generalization by the following example.
\begin{example}
Consider the following nonlinear   equation with Dirichlet boundary condition:
\[
\begin{cases}
& - \Delta u_\eps+u_\eps-u_\eps^3=f \quad \text{in } \dom_\eps,\\
& u_\eps=0 \quad \text{on } \partial\dom_\eps,
\end{cases}
\]
Assume  the ans\"{a}tz  as before
$
u_\eps=\sum_{n=0}^\infty \eps^n u_n,
$
then
\[
\sum_{n=0}^\infty \eps^n(-\Delta u_n+u_n)-\biggl(\sum_{n=0}^\infty \eps^n u_n\biggr)^3=f.
\]
Successively equating coefficients of like powers $\eps^n$ yields
\[
-\Delta u_0+u_0-u_0^3=f,
\]
and for $n\geqslant 1$,
\[
-\Delta u_n+u_n-3u_0^2u_n=\sum_{{i_0,i_1,\cdots,i_{n-1}\geqslant 0}\atop{{i_0+i_1+\cdots+i_{n-1}=3}\atop{i_1+2i_2+\cdots+(n-1)i_{n-1}}=n}}
\frac{3!}{i_0!~i_1!\cdots i_{n-1}!}u_0^{i_0}u_1^{i_1}\cdots u_{n-1}^{i_{n-1}}.
\]
In particular, the equations for the first few terms are
\begin{align}
&-\Delta u_1+u_1-3u_0^2u_1=0,\\
&-\Delta u_2+u_2-3u_0^2u_2=3u_0u_1^2,\\
&-\Delta u_3+u_3-3u_0^2u_3=6u_0u_1u_2+u_1^3.
\end{align}
Note that all these equations for $n\geq 1$ are linear.
The boundary conditions on $\partial\dom$ for each $u_n$ are
exactly the same as in \eqref{PDE:un}.

\end{example}

\section{Collection of proofs}
\label{sec:proof}

\subsection{Proof of Theorem \ref{thm:error} }\label{sec:proof1}

The proof will rely on the following result:
\begin{lemma}\label{lem:est}
Let us assume in addition to  Assumption \ref{asmp}, that $\Lo w=0$ in $\dom$.
Then
\[
\norm{w}_{H^m(\dom)} \leq C \Bigl((1-\delta_{0,m})\norm{w}_{L^2(\dom)}+\norm{w}_{H^{m}(\partial\dom)}\Bigr),
\]
where $C$ depends only on $d$, $m$, $\partial\dom$ and the coefficients $a^{ij}$, $b^i$, $c$.
\end{lemma}
\begin{proof}
By Theorem 3 in \cite{Marschall:1987}, there exists $\phi\in H^{m}(\dom)$
for which $\phi=w$ on $\partial\dom$ and
\[
\norm{\phi}_{H^{m}(\dom)}\leq C\norm{w}_{H^{m-\frac12}(\partial\dom)}\leq C\norm{w}_{H^{m}(\partial\dom)}.
\]
Then using Corollary 8.7 and Theorem 8.13 in \cite{Gilbarg:2001}, we have
\[
\norm{w}_{H^m(\dom)} \leq C \Bigl((1-\delta_{0,m})\norm{w}_{L^2(\dom)}+\norm{\phi}_{H^{m}(\dom)}\Bigr).
\]
Combining the above two inequalities yields the asserted inequality.
\end{proof}

\begin{proof}[Proof of Theorem \ref{thm:error}]
The boundary condition $u_\eps=g$ on $\partial{\dom_\eps}$ can be rewritten as
\[
u_\eps\bigl(\bx+\eps h(\bx)\bn(\bx)\bigr)=g\bigl(\bx+\eps h(\bx)\bn(\bx)\bigr),\quad \forall \bx \in \partial\dom.
\]
By the Taylor expansion to the $(n+1)$-th order, we have
\[
u_{\eps}(\bx)=\sum_{k=0}^n\frac{\eps^k\bigl(h(\bx)\bigr)^k}{k!}{\partial_{\bn}^k}g(\bx) -\sum_{k=1}^n\frac{\eps^k\bigl(h(\bx)\bigr)^k}{k!}{\partial_{\bn}^k}u_\eps(\bx)+\Od(\eps^{n+1}),\quad \bx\in\partial\dom.
\]
Thus by using \eqref{eqn:vnbc}, we find on $\partial\dom$
\begin{equation}\label{eqn:vn-ueps}
v^{[n]}-u_\eps=\sum_{k=1}^n\frac{\eps^k\bigl(h(\bx)\bigr)^k}{k!}{\partial_{\bn}^k}\bigl(u_\eps-v^{[n-k]}\bigr)+\Od(\eps^{n+1}).
\end{equation}
Then \eqref{eqn:vn err} can be proved by induction on $n$. Actually, for $n=0$, \eqref{eqn:vn err} follows from
\eqref{eqn:vn-ueps} and Lemma \ref{lem:est}.
Now suppose we have proved \eqref{eqn:vn err} for all $k<n$, then  by trace inequality, we would have
\[
\norm{\partial_{\bn}^k\bigl(u_\eps-v^{[n-k]}\bigr)}_{H^m(\partial\dom)}=\Od(\eps^{n-k+1}),\quad \forall m\geq 0.
\]
Plugging these into \eqref{eqn:vn-ueps} yields
\begin{equation*}\label{eqn:vn err bdry}
\bigl\| v^{[n]} -u_\eps \bigr\|_{H^m(\partial\dom)} = \Od (\eps^{n+1}).
\end{equation*}
Therefore \eqref{eqn:vn err} follows by Lemma \ref{lem:est}. This completes the induction and
the proof of \eqref{eqn:vn err}.
\end{proof}

\subsection{Proof of Lemma \ref{lem:dnuetodnui}}\label{sec:proof2}
\begin{proof}
Let $\vec e_i$, $i=1,\cdots, d$, denote the standard basis for $\RR^d$.
We first note that every partial derivative $\partial_{x_i}$, $i=1,\cdots, d$,  can be expressed
in terms of the unit normal vector $\bn=(n_1,\cdots,n_d)$,
the normal derivative $\partial_{\bn}$ and a tangential derivative $\partial_{\btau_i}$ along a certain tangent vector
$\btau_i=\vec e_i-n_i\bn$. In fact,
it is clear that every $\btau_i$, $i=1,\cdots, d$, is a tangent vector, since $\btau_i\cdot \bn=0$. Recall the elementary facts that $\partial_{x_i}$ may be understood as the directional derivative $\partial_{\vec e_i}$
and that the directional derivative $\partial_{\vec v}$ is a linear functional of a direction vector $\vec v$.
Thus we deduce
\[
\partial_{x_i}=\partial_{\vec e_i}=n_i\partial_{\bn}+\partial_{\btau_i}.
\]
Then \eqref{eqn:tranconddnun} can be written as
\[
\begin{split}
&\sum_{i,j=1}^d a^{ij}_\rin n_in_j\partial_{\bn} u_{\rin,n}+
\sum_{i,j=1}^d a^{ij}_\rin n_i\partial_{\btau_j} u_{\rin,n}\\
&= \sum_{i,j=1}^d a^{ij}_\ext n_in_j\partial_{\bn}u_{\ext,n}
+\sum_{i,j=1}^d a^{ij}_\ext n_i\partial_{\btau_j} u_{\ext,n}.
\end{split}
\]
By \eqref{eqn:trancondun}, $\partial_{\btau_j}u_{\ext,n}$ in the last sum amounts to $\partial_{\btau_j}u_{\rin,n}$.
Thus we can solve
\[
\partial_{\bn}u_{\ext,n}=\frac{\displaystyle\sum_{i,j=1}^d a^{ij}_\rin n_in_j\partial_{\bn} u_{\rin,n}+
\sum_{i,j=1}^d \left(a^{ij}_\rin-a^{ij}_\ext\right) n_i\partial_{\btau_j} u_{\rin,n}}{\displaystyle\sum_{i,j=1}^d a^{ij}_\ext n_in_j}.
\]
Note that  $\sum_{i,j=1}^d a^{ij}_\ext n_in_j\neq 0$ due to
 the elliptic condition \eqref{eqn:ellipcond}. The proof is complete.
\end{proof}

\subsection{Proof of  \eqref{eqn:intdeldomgen}}\label{sec:proof3}
To compute $\abs{J(\bxi,t)}$, we make use of the  Weingarten equations  in the differential geometry of hypersurfaces \cite{do1976differential}, which give
the linear expansions of  the derivatives of the unit normal vector $\bn$ to the hypersurface $\Gamma$
in  terms of the tangent vectors $\partial_{\xi_j}\btheta$, $j=1,\cdots, d-1$:
\begin{equation}\label{eqn:Weingarten}
\partial_{\xi_i}{\bn}=\sum_{j=1}^{d-1}W_i^{~j}\partial_{\xi_j}\btheta,  \quad i=1,\cdots, d-1,
\end{equation}
where $W=(W_i^{~j})$ is the matrix representation
of the so called shape operator or Weingarten map $\tensor W$,
and is given  
by
\[
W_i^{~j}=-\sum_{k=1}^{d-1}h_{ik}g^{kj},
\]
where $(h_{ik})$ is the matrix representation of the second fundamental form, and $(g^{kj})$
 is the inverse of the matrix representation $(g_{kj})$ of the first fundamental form,
and all the above matrix representations are with respect to the basis $\partial_{\xi_i}\btheta$, $i=1,\cdots, d-1$.
By \eqref{eqn:Weingarten}, we compute
\[
\partial_{\xi_i}\bx=\partial_{\xi_i}\btheta+t\partial_{\xi_i}{\bn}=\sum_{j=1}^{d-1}(\delta_i^{~j}+tW_i^{~j})\partial_{\xi_j}\btheta, \quad i=1,\cdots, d-1,
\]
where $\delta_i^{~j}$ is equal to 1 if $i=j$ and 0 otherwise. Thus we obtain for  sufficiently small $\abs{t}$,
\[
\begin{split}
\abs{J(\bxi,t)}\rd\bxi=&\abs{\det(I+tW)}\abs{\det(\partial_{\xi_1}\bx,\cdots,\partial_{\xi_{d-1}}\bx,\bn)}\rd\bxi\\
=&\det(I+tW)\rd S_{\Gamma}(\btheta),
\end{split}
\]
where $I$ denotes the identity matrix, and $\rd S_{\Gamma}(\btheta)$ denotes  the surface area element on the hypersurface $\Gamma$.
Consequently, \eqref{eqn:chofvar} becomes
\begin{equation} \int_{\delta\dom_\eps} F(\bx)\,\rd\bx=\int_{\Gamma}\biggl(\int_0^{\eps h(\btheta)}\tilde{F}(\btheta,t)\det(I+tW) \,\rd t\biggr)\rd S_\Gamma({\btheta}),
\end{equation}
where $\tilde{F}(\btheta,t):=F\bigl(\btheta+t\bn(\btheta)\bigr)$, $\btheta\in\Gamma$.

\section{Explicit formula of boundary conditions for low order terms
of two-parameter expansion in Section \ref{sec:4}}
\label{sec:case}

$\set{u_{\rin,m,n}}$ satisfy the equations
\[
-\Delta u_{\rin,m,n}=\delta_{0,m}\delta_{0,n}f_\rin \quad \text{in } {\dom}.
\]
Their boundary conditions for a few lower order are listed below.

\subsection{Case (i)}

The boundary conditions of the expansion $\set{u_{\rin,m,n}}$ on $\partial \dom$ for  the first a few terms ($m+n\leq 2$)  are listed below:
\begin{align*}
& u_{\rin,0,0}=g,~~\qquad ~~ u_{\rin,1,0}=h\partial_{\bn}g,\\
& u_{\rin,2,0}=\dfrac{h^2}2\partial_{\bn}^2g-\dfrac{h^2}2 F_{2,0}[g,0],
\end{align*}
and
\begin{align*}
& u_{\rin,0,1}=-h\partial_{\bn}u_{\rin,0,0},
~~\qquad~~ u_{\rin,0,2}=-h\partial_{\bn}u_{\rin,0,1},\\
& u_{\rin,1,1}
=-h\partial_{\bn}u_{\rin,1,0}-\dfrac{h^2}2F_{2,f_\ext}[0,\partial_{\bn}u_{\rin,0,0}].
\end{align*}

\subsection{Case (ii)$_1$}
The Neumann boundary conditions on $\partial \dom$ for $u_{\rin,m,n}$ with $m+n\leq 2$ read
\begin{align*}
& \partial_{\bn}u_{\rin,0,0}=0,
\qquad \partial_{\bn}u_{\rin,0,1}=\frac gh-\frac1h u_{\rin,0,0},\\
& \partial_{\bn}u_{\rin,0,2}=-\frac1h u_{\rin,0,1},
\qquad \partial_{\bn}u_{\rin,1,0}
=-\dfrac{h}{2}F_{2,f_\ext}[0,0],\\
& \partial_{\bn}u_{\rin,2,0}=-\dfrac{h}{2}F_{2,0}
[0,\partial_{\bn}u_{\rin,1,0}]-\dfrac{h^2}{6}F_{3,f_\ext}
[0,0],
\\
& \partial_{\bn}u_{\rin,1,1}=\partial_{\bn}g-\frac1h u_{\rin,1,0}-\dfrac{h}{2}F_{2,0}[0,\partial_{\bn}u_{\rin,0,1}].
\end{align*}

\subsection{Case (ii)$_2$}
The Neumann boundary conditions on $\partial\dom$ for $u_{\rin,m,n}$ with $m+n\leq 2$ read
\begin{align*}
& \partial_{\bn}u_{\rin,0,-1}=\partial_{\bn}u_{\rin,1,-1}=\partial_{\bn}u_{\rin,2,-1}=\partial_{\bn}u_{\rin,3,-1}=0,\\
& \partial_{\bn}u_{\rin,0,0}=-\dfrac 1h u_{\rin,0,-1},
\qquad
 \partial_{\bn}u_{\rin,0,1}=\frac gh-\dfrac 1h u_{\rin,0,0},
 \qquad \partial_{\bn}u_{\rin,0,2}=-\dfrac 1h u_{\rin,0,1},\\
& \partial_{\bn}u_{\rin,1,0}=-\dfrac 1h u_{\rin,1,-1}-\frac h2 F_{2,f_\ext}[0,\partial_{\bn}u_{\rin,0,0}],\\
& \partial_{\bn}u_{\rin,2,0}=-\dfrac 1h u_{\rin,2,-1}-\frac h2 F_{2,0}[u_{\rin,0,-1},\partial_{\bn}u_{\rin,1,0}]
-\frac {h^2}6 F_{3,f_\ext}[0,\partial_{\bn}u_{\rin,0,0}],\\
& \partial_{\bn}u_{\rin,1,1}=\partial_{\bn}g-\frac 1hu_{\rin,1,0}-\frac h2 F_{2,0}[0,\partial_{\bn}u_{\rin,0,1}].
\end{align*}

The corresponding solvability conditions to give the  the unique solutions $u_{\rin,m,n+1}$ are the following.
\begin{align*}
& \int_{\partial\dom} \dfrac {u_{\rin,0,-1}}h =\int_\dom f_\rin,
\qquad
\int_{\partial\dom } \frac {u_{\rin,0,0}}h =\int_{\partial\dom } \frac {g}h,
\qquad
\int_{\partial\dom} \dfrac {u_{\rin,0,1}}h =0,\\
& \int_{\partial\dom} \dfrac {u_{\rin,1,-1}}h=-\int_{\partial\dom} \frac h2 F_{2,f_\ext}[0,\partial_{\bn}u_{\rin,0,0}],\\
& \int_{\partial\dom} \dfrac {u_{\rin,1,0}}h=\int_{\partial\dom}\partial_{\bn}g-\int_{\partial\dom}\frac h2 F_{2,0}[0,\partial_{\bn}u_{\rin,0,1}],\\
& \int_{\partial\dom} \dfrac {u_{\rin,2,-1}}h=-\int_{\partial\dom}\frac h2 F_{2,0}[u_{\rin,0,-1},\partial_{\bn}u_{\rin,1,0}]
-\int_{\partial\dom}\frac {h^2}6 F_{3,f_\ext}[0,\partial_{\bn}u_{\rin,0,0}].
\end{align*}

\subsection{Case (iii)}
The  Robin boundary conditions on $\partial\dom$ for $u_{\rin,m,n}$ with $m+n\leq 2$
have the following expressions:
\begin{align*}
& u_{\rin,0,0}+ch\partial_{\bn}u_{\rin,0,0}=g,\\
& u_{\rin,0,1}+ch\partial_{\bn}u_{\rin,0,1}=-h\partial_{\bn}u_{\rin,0,0},\\
& u_{\rin,0,2}+ch\partial_{\bn}u_{\rin,0,2}=-h\partial_{\bn}u_{\rin,0,1},\\
& u_{\rin,1,0}+ch\partial_{\bn}u_{\rin,1,0}
=h\partial_{\bn}g-\frac{h^2}2F_{2,cf_\ext}[0,c\partial_{\bn}u_{\rin,0,0}],\\
& u_{\rin,2,0}+ch\partial_{\bn}u_{\rin,2,0}=\frac{h^2}2\partial_{\bn}^2g
-\frac{h^2}2F_{2,0}[u_{\rin,0,0},c\partial_{\bn}u_{\rin,1,0}]
-\frac{h^3}{6}F_{3,cf_\ext}[0,c\partial_{\bn}u_{\rin,0,0}], \\
& u_{\rin,1,1}+ch\partial_{\bn}u_{\rin,1,1}=-h\partial_{\bn}u_{\rin,1,0}-\dfrac{h^2}2F_{2,f_\ext}[0,\partial_{\bn}u_{\rin,0,0}+c\partial_{\bn}u_{\rin,0,1}].
\end{align*}

\section{Case (ii) in Section \ref{sec:5}}

For Case (ii), we treat $\eps$ and $\sigma$ as independent small parameters.
Therefore we assume  $u^+_\eps$ and $u^-_\eps$  have double asymptotic expansions
\[
u^+_\eps(\bx)=\sum_{m,n=0}^\infty u^+_{m,n}(\bx)\eps^m\sigma^n,  \quad \bx\in\dom^+,
\]
\[
u^-_\eps(\bx)=\sum_{m=0}^\infty\sum_{n=-1}^{\infty} u^-_{m,n}(\bx)\eps^m\sigma^n, \quad \bx\in \dom^-.
\]
Note that for $u^-_\eps(\bx)$, the terms $\sigma^n$ start from $n=-1$.
Define
\begin{equation}\label{eqn:subansatz+}
u^+_{m,\cdot}(\bx)=\sum_{n=0}^\infty u^+_{m,n}(\bx)\sigma^n, \quad \bx\in\dom^+,
\end{equation}
\begin{equation}\label{eqn:subansatz-}
u^-_{m,\cdot}(\bx)=\sum_{n=-1}^{\infty} u^-_{m,n}(\bx)\sigma^n, \quad \bx\in \dom^-.
\end{equation}

Inserting \eqref{eqn:subansatz+}\eqref{eqn:subansatz-} into the last three equations
and  collecting terms with equal powers of $\sigma$, we obtain on $\Gamma$
\begin{align}
&u^-_{0,-1}=0,\label{eqn:u^-_{0,-1}bdc}\\
&u^-_{0,n}=u^+_{0,n},   \quad n\geq 0,\label{eqn:u^-_{0,n}bdcapp}\\
&u^-_{1,-1}=-h\partial_{\bn}u^-_{0,-1},\label{eqn:u^-_{1,-1}bdc}\\
&u^-_{1,n}=u^+_{1,n}+h\sqbk{{\partial_{\bn}}u_{0,n}}  \quad n\geq 0;\label{eqn:u^-_{1,n}bdc}
\end{align}
and for $n\geq 0$, and all $v\in H_0^1(\dom)$
\begin{align}
&\int_{\dom^+} \nabla u^+_{0,n} \cdot \nabla v \,\rd\bx +
\int_{\dom^-} \nabla u^-_{0,n-1} \cdot \nabla v \,\rd\bx=
\delta_{0,n}\int_{\dom} f v\,\rd\bx;\label{eqn:u_{0,n}weak}\\
&\int_{\dom^+} \nabla u^+_{1,n} \cdot \nabla v \,\rd\bx +
\int_{\dom^-} \nabla u^-_{1,n-1} \cdot \nabla v \,\rd\bx\nonumber\\
&- \int_{\Gamma}h\bigl[(\nabla_{\Gamma} u^+_{0,n}-\nabla_{\Gamma} u^-_{0,n-1}) \cdot \nabla_\Gamma v
+(\partial_{\bn}u^+_{0,n}- \partial_{\bn} u^-_{0,n-1})\partial_{\bn} v\bigr]\,\rd S_{\Gamma}=0. \label{eqn:u_{1,n}weak}
\end{align}
Therefore we deduce from \eqref{eqn:u^-_{0,-1}bdc},  \eqref{eqn:u^-_{0,n}bdc} and \eqref{eqn:u_{0,n}weak}
that: $u^-_{0,-1}$ satisfies
\[
\begin{cases}
& -\Delta u^-_{0,-1}=f \quad \text{in } \dom^-,\\
& u^-_{0,-1}=0 \quad \text{on }\Gamma,\\
& u^-_{0,-1}=0\quad \text{on }\partial\dom^-\cap\partial\dom.
\end{cases}
\]
For $n\geq 0$, $u^+_{0,n}$ satisfies
\[
\begin{cases}
& -\Delta u^+_{0,n}=\delta_{0,n}f \quad \text{in } \dom^+,\\
& \partial_{\bn}u^+_{0,n}=\partial_{\bn}u^-_{0,n-1} \quad \text{on }\Gamma,\\
& u^+_{0,n}=\delta_{0,n}g\quad \text{on }\partial\dom^+\cap\partial\dom;
\end{cases}
\]
and   $u^-_{0,n}$ satisfies
\[
\begin{cases}
& -\Delta u^-_{0,n}=0 \quad \text{in } \dom^-,\\
& u^-_{0,n}=u^+_{0,n} \quad \text{on }\Gamma,\\
& u^-_{0,n}=\delta_{0,n}g\quad \text{on }\partial\dom^-\cap\partial\dom.
\end{cases}
\]
Similarly, from \eqref{eqn:u^-_{1,-1}bdc},  \eqref{eqn:u^-_{1,n}bdc} and   \eqref{eqn:u_{1,n}weak},
we obtain that
 $u^-_{1,-1}$ satisfies
\[
\begin{cases}
& -\Delta u^-_{1,-1}=0 \quad \text{in } \dom^-,\\
& u^-_{1,-1}=-h\partial_{\bn}u^-_{0,-1} \quad \text{on }\Gamma,\\
& u^-_{1,-1}=0\quad \text{on }\partial\dom^-\cap\partial\dom.
\end{cases}
\]
For $n\geq 0$, $u^+_{1,n}$ satisfies
\[
\begin{cases}
& -\Delta u^+_{1,n}=0 \quad \text{in } \dom^+,\\
& \partial_{\bn}u^+_{1,n}=\partial_{\bn}u^-_{1,n-1}-\nabla_{\Gamma}\cdot
 \bigl[h(\nabla_{\Gamma} u^+_{0,n}-\nabla_{\Gamma} u^-_{0,n-1})\bigr] \quad \text{on }\Gamma,\\
& u^+_{1,n}=0\quad \text{on }\partial\dom^+\cap\partial\dom;
\end{cases}
\]
and   $u^-_{1,n}$ satisfies
\[
\begin{cases}
& -\Delta u^-_{1,n}=0 \quad \text{in } \dom^-,\\
& u^-_{1,n}=u^+_{1,n}+h(\partial_{\bn}u^+_{0,n}-\partial_{\bn}u^-_{0,n}) \quad \text{on }\Gamma,\\
& u^-_{1,n}=0\quad \text{on }\partial\dom^-\cap\partial\dom.
\end{cases}
\]

In this Case (ii),  there is no emergence of the pure Neumann boundary value problem,
even for the  situation  in the right panel in Figure \ref{fig:interface perturbation}.

\end{document}